\documentclass[reqno,11pt]{amsart}
\usepackage[arxiv]{optional}
\pdfoutput=1

\usepackage{filecontents}

\newcommand{\acmarxiv}[2]{\opt{acm}{#1}\opt{arxiv}{#2}}

\usepackage{mathpartir}
\usepackage{enumerate}
\usepackage{chngcntr} 
\usepackage{xcolor}
\definecolor{darkgreen}{rgb}{0,0.45,0}
\definecolor{darkred}{rgb}{0.75,0,0}
\definecolor{darkblue}{rgb}{0,0,0.6}

\usepackage[T1]{fontenc}
\usepackage[utf8]{inputenc}
\acmarxiv
  {\usepackage{amsmath, amssymb}}
  {\usepackage{amssymb}
   \usepackage{libertine}
   \usepackage{microtype}}

\usepackage[colorlinks,citecolor=darkgreen,linkcolor=darkred,urlcolor=darkblue,linktoc=section]{hyperref}
\usepackage{breakurl}

\acmarxiv
  {\input{temp-local-universes-acm-thmenvs}}
  {\input{temp-local-universes-arxiv-thmenvs}
   \input{temp-local-universes-arxiv-secenvs}}

\usepackage{tikz}
\usetikzlibrary{fit}
\usetikzlibrary{shapes}
\usetikzlibrary{arrows}
\usetikzlibrary{decorations.markings}
\usetikzlibrary{positioning}

\tikzset{mylabel/.style={font=\footnotesize}}
\tikzset{mylabelsmall/.style={font=\tiny}}
\tikzset{mymidlabel/.style={auto=false,fill=white,font=\footnotesize}}
\tikzset{cross line/.style={preaction={draw=white,-,line width=6pt}}}

\usetikzlibrary{matrix,arrows}

\pgfarrowsdeclare{xyto}{xyto}
{
  \arrowsize=0.22pt
  \advance\arrowsize by .7\pgflinewidth  
  \pgfarrowsleftextend{-10\arrowsize-.5\pgflinewidth}
  \pgfarrowsrightextend{.5\pgflinewidth}
}
{
  \arrowsize=0.22pt
  \advance\arrowsize by .7\pgflinewidth  
  \pgfsetdash{}{0pt} 
  \pgfsetroundjoin	
  \pgfsetroundcap	
  \pgfpathmoveto{\pgfpointorigin}
  \pgfpathcurveto
     {\pgfpoint{-4.2\arrowsize}{0.6\arrowsize}}
     {\pgfpoint{-7.2\arrowsize}{1.5\arrowsize}}
     {\pgfpoint{-10\arrowsize}{3.1\arrowsize}}
  \pgfusepathqstroke
  \pgfpathmoveto{\pgfpointorigin}
  \pgfpathcurveto
     {\pgfpoint{-4.2\arrowsize}{-0.6\arrowsize}}
     {\pgfpoint{-7.2\arrowsize}{-1.5\arrowsize}}
     {\pgfpoint{-10\arrowsize}{-3.1\arrowsize}}
  \pgfusepathqstroke
}

\newlength{\ontotipoffset}

\pgfarrowsdeclare{xyonto}{xyonto}
{
  \arrowsize=0.22pt
  \advance\arrowsize by .7\pgflinewidth  
  \pgfarrowsleftextend{-10\arrowsize-.5\pgflinewidth}
  \pgfarrowsrightextend{.5\pgflinewidth}
}
{
  \arrowsize=0.22pt
  \advance\arrowsize by .7\pgflinewidth  
  \ontotipoffset=-6\arrowsize
  \pgfsetdash{}{0pt} 
  \pgfsetroundjoin	
  \pgfsetroundcap	
  \pgfpathmoveto{\pgfpointorigin}
  \pgfpathcurveto
     {\pgfpoint{-4.2\arrowsize}{0.6\arrowsize}}
     {\pgfpoint{-7.2\arrowsize}{1.5\arrowsize}}
     {\pgfpoint{-10\arrowsize}{3.1\arrowsize}}
  \pgfusepathqstroke
  \pgfpathmoveto{\pgfpointorigin}
  \pgfpathcurveto
     {\pgfpoint{-4.2\arrowsize}{-0.6\arrowsize}}
     {\pgfpoint{-7.2\arrowsize}{-1.5\arrowsize}}
     {\pgfpoint{-10\arrowsize}{-3.1\arrowsize}}
  \pgfusepathqstroke
  \pgfpathmoveto{\pgfpoint{\ontotipoffset}{0}}
  \pgfpathcurveto
     {\pgfpoint{\ontotipoffset-4.2\arrowsize}{0.6\arrowsize}}
     {\pgfpoint{\ontotipoffset-7.2\arrowsize}{1.5\arrowsize}}
     {\pgfpoint{\ontotipoffset-10\arrowsize}{3.1\arrowsize}}
  \pgfusepathqstroke
  \pgfpathmoveto{\pgfpoint{\ontotipoffset}{0}}
  \pgfpathcurveto
     {\pgfpoint{\ontotipoffset-4.2\arrowsize}{-0.6\arrowsize}}
     {\pgfpoint{\ontotipoffset-7.2\arrowsize}{-1.5\arrowsize}}
     {\pgfpoint{\ontotipoffset-10\arrowsize}{-3.1\arrowsize}}
  \pgfusepathqstroke
}

\pgfarrowsdeclare{xytri}{xytri}
{
  \arrowsize=0.22pt
  \advance\arrowsize by .7\pgflinewidth  
  \pgfarrowsleftextend{-.5\pgflinewidth}
  \pgfarrowsrightextend{10\arrowsize+.5\pgflinewidth}
}
{
  \arrowsize=0.22pt
  \advance\arrowsize by .7\pgflinewidth  
  \pgfsetdash{}{0pt} 
  \pgfsetroundjoin	
  \pgfsetroundcap	
  \pgfpathmoveto{\pgfpoint{10\arrowsize}{0}}
  \pgfpathcurveto
     {\pgfpoint{5.8\arrowsize}{0.6\arrowsize}}
     {\pgfpoint{2.8\arrowsize}{1.5\arrowsize}}
     {\pgfpoint{0\arrowsize}{3.1\arrowsize}}
  \pgfpathmoveto{\pgfpoint{10\arrowsize}{0}}
  \pgfpathcurveto
     {\pgfpoint{5.8\arrowsize}{-0.6\arrowsize}}
     {\pgfpoint{2.8\arrowsize}{-1.5\arrowsize}}
     {\pgfpoint{0\arrowsize}{-3.1\arrowsize}}
  \pgflineto{\pgfpoint{0}{3.1\arrowsize}}
  \pgfusepathqstroke
}
\tikzset{>=xyto} 
\pgfarrowsdeclarealias{<|}{|>}{xytri}{xytri}

\newcommand{\myto}[1][]{ \mathrel{
  \tikz[baseline={([yshift=-0.55ex]a.south)}]{%
    \node[minimum width=1em,align=center,inner xsep=0.3ex,inner ysep=0.15ex] (a) {$\scriptstyle #1$};
    \draw[->] (a.south west) -- ([xshift=0.6ex]a.south east);}
  \mkern-1mu}}

\newcommand{\fibto}{\myto}

\DeclareFontFamily{U}{MnSymbolE}{}
\DeclareFontShape{U}{MnSymbolE}{m}{n}{
    <-6>  MnSymbolE5
   <6-7>  MnSymbolE6
   <7-8>  MnSymbolE7
   <8-9>  MnSymbolE8
   <9-10> MnSymbolE9
  <10-12> MnSymbolE10
  <12->   MnSymbolE12}{}

\newcommand{\mnquote}[1]{\usefont{U}{MnSymbolE}{m}{n}\char#1\relax}

\newcommand{\quinelquote}{\raisebox{.1ex}{\rlap{\mnquote{'036}}\kern.2em}}
\newcommand{\quinerquote}{\/\raisebox{.1ex}{\kern.2em\llap{\mnquote{'043}}}}

\newcommand{\T}{\mathbb{T}} 

\newcommand{\CC}{\mathcal{C}}
\newcommand{\EE}{\mathcal{E}}
\newcommand{\FF}{\mathcal{F}}
\newcommand{\GG}{\mathcal{G}}
\newcommand{\JJ}{\mathcal{J}}

\newcommand{\TT}{\mathcal{T}}

\newcommand{\RR}{\mathcal{R}}

\newcommand{\LL}{\mathcal{L}}

\newcommand{\cat}{\textbf{Cat}}

\newcommand{\gpd}{\textbf{Gpd}}
\newcommand{\sets}{\textbf{Sets}}
\newcommand{\ssets}{\textbf{SSets}}

\newcommand{\Fib}{\mathbf{Fib}}
\newcommand{\SplFib}{\mathbf{SplFib}}
\newcommand{\CompCat}{\mathbf{CompCat}}
\newcommand{\SplCompCat}{\mathbf{SplCompCat}}

\newcommand{\clov}{\mathrm{cl}}

\newcommand{\arr}{\rightarrow} 

\DeclareMathOperator{\cod}{cod}

\newcommand{\myequiv}{\simeq}
\newcommand{\homot}{\sim}

\acmarxiv
  {\newcommand{\myemph}[1]{\textbf{#1}}}
  {\newcommand{\myemph}[1]{\emph{#1}}}
\newcommand{\iso}{\cong}
\newcommand{\comp}{\circ}
\newcommand{\famop}{\triangleleft}
\DeclareMathOperator{\ev}{ev}
\newcommand{\Imp}{\Rightarrow}
\newcommand{\adjoint}{\dashv}
\newcommand{\fibslice}{\mathord{/ \! /}}

\newcommand{\LF}{\hyperref[def:LF]{(LF)}}

\newcommand{\pdfCbang}{\texorpdfstring{$\CC_!$}{C\_!}}

\newcommand{\types}{\vdash}
\newcommand{\judge}[3][]{#2\;\vdash_{#1}\;#3}
\newcommand{\of}{\mathord{:}}
\acmarxiv
  {\newcommand{\fibration}{\ \syntax{Fib}}}
  {\newcommand{\fibration}{\ \syntax{fib}}}
\acmarxiv
  {\newcommand{\type}{\ \syntax{Type}}}
  {\newcommand{\type}{\ \syntax{type}}}

\newcommand{\lscott}{[\![}
\newcommand{\rscott}{]\!]}
\newcommand{\interp}[1]{\lscott #1 \rscott}

\newcommand{\lexinterp}{[} 
\newcommand{\rexinterp}{]} 
\newcommand{\exinterp}[1]{\lexinterp #1 \rexinterp}
\newcommand{\exprod}{\overline{\prod}}
\newcommand{\exapp}{\overline{\app}}

\acmarxiv{\let\syntax\mathtt}{\let\syntax\mathsf}

\newcommand{\app}{\syntax{app}}

\newcommand{\id}[1][]{\syntax{Id}_{#1}}
\newcommand{\Jelim}[1][]{\syntax{J}_{#1}}
\newcommand{\jj}{\syntax{j}}
\newcommand{\rr}{\syntax{r}}

\newcommand{\dpair}[2]{\langle #1,#2 \rangle}
\newcommand{\dsplit}[2]{\syntax{split}_{#1,#2}}

\newcommand{\inl}{\syntax{inl}}
\newcommand{\inr}{\syntax{inr}}
\newcommand{\copair}[2]{\langle #1 , #2 \rangle}

\newcommand{\W}{\syntax{W}}
\newcommand{\fold}{\syntax{fold}} 
\newcommand{\Welim}{\syntax{wrec}}

\newcommand{\tensor}{\otimes}

\newcommand{\zero}{\syntax{0}}
\newcommand{\unit}{\syntax{1}}
\newcommand{\pt}{\syntax{tt}}
\newcommand{\unitelim}{\syntax{urec}}
\newcommand{\univ}{\syntax{V}}
\newcommand{\elt}[1][]{\syntax{el}_{#1}}

\newcommand{\baseof}[1]{{V_{#1}}}
\newcommand{\topof}[1]{E_{#1}}
\newcommand{\nameof}[1]{\quinelquote #1 \quinerquote}

\acmarxiv
  {}
  {}

\newcommand{\stable}[2]{#2}

\begin{document}

\opt{acm}{\markboth{P. LeF. Lumsdaine and M. A. Warren}{The local universes model}}

\acmarxiv
  {\title{The local universes model: an overlooked coherence construction for dependent type theories}}
  {\title[The local universes model]{The local universes model: \\ an overlooked coherence construction \\ for dependent type theories}}

\acmarxiv
  {\input{temp-local-universes-acm-authors}}
  {\input{temp-local-universes-arxiv-authors}}

\opt{arxiv}{\date{\today}}

\begin{abstract}
We present a new coherence theorem for comprehension categories, providing strict models of dependent type theory with all standard constructors, including dependent products, dependent sums, identity types, and other inductive types.

Precisely, we take as input a ``weak model'': a comprehension category, equipped with structure corresponding to the desired logical constructions.
We assume throughout that the base category is close to locally Cartesian closed: specifically, that products and certain exponentials exist.
Beyond this, we require only that the logical structure should be \emph{\stable{3}{weakly stable}}---a pure existence statement, not involving any specific choice of structure, weaker than standard categorical Beck--Chevalley conditions, and holding in the now standard homotopy-theoretic models of type theory.

Given such a comprehension category, we construct an equivalent split one, whose logical structure is strictly stable under reindexing.
This yields an interpretation of type theory with the chosen constructors.

The model is adapted from Voevodsky’s use of universes for coherence, and at the level of fibrations is a classical construction of Giraud.
It may be viewed in terms of local universes or delayed substitutions.
\end{abstract}

\opt{acm}{\input{temp-local-universes-acm-metadata}}

\opt{arxiv}{\ \vspace{-1.5ex}}

\maketitle

\opt{arxiv}{\vspace{-3ex}
            \tableofcontents}

\section{Introduction}

Constructing interpretations of dependent type theory from scratch is a laborious, bureaucratic, and error-prone task.
Various algebraic axiomatizations of such models (contextual categories and their relatives) abstract away many of the technicalities, allowing constructions of models to concentrate, for the most part, on their real substance---the logical constructions which genuinely constitute the model.

One issue remains, however, which feels like it ought to be another mere technicality, but which has not been so successfully abstracted away: the so-called coherence problem.
In most concrete models, it is not hard to resolve directly, by more or less ad hoc methods.
In a few, however (notably, Voevodsky’s simplicial model \cite{kapulkin-lumsdaine-voevodsky:simplicial-model} and similar homotopy-theoretic models), it is much less straightforward to deal with.
It has also proven problematic for more abstract theorems on model existence, in for instance the weak factorization system models of \cite{warren-awodey:homotopy-theoretic-models}.

The most powerful coherence result to date is that of \acmarxiv{\cite{hofmann:lcccs}}{Hofmann \cite{hofmann:lcccs}}, which provides coherence for a wide range of models.
However, it does not apply to the homotopy-theoretic models mentioned above.
The present paper gives a new coherence construction, with slightly different hypotheses from Hofmann’s, applying in particular to a wide range of homotopy-theoretic models of intensional type theory. \\

Specifically, the present work arose from a careful reading of Voevodsky’s model in simplicial sets \cite{voevodsky:notes-2011-04,kapulkin-lumsdaine-voevodsky:simplicial-model}.
Universes are used there for two distinct purposes: firstly to obtain coherence of the model; and secondly to become type-theoretic universes within the model.
It turned out that not only may the two aspects be entirely disentangled, but moreover, the coherence construction may be modified to work without a universe.

The precise sense of \myemph{coherence} in question is that substitution/pullback should be strictly functorial, and should strictly preserve all the logical structure under consideration.
This holds in the syntax of type theory; so it must hold in any direct model of the theory.

However, categorically such strictness is rather unnatural, involving on-the-nose equality of objects, and rarely arises automatically in nature.
We of course expect to find \myemph{some} kind of stability---some preservation of logical structure under pullback, such as stability up to isomorphism, or similar---and expect this to be necessary for modelling the type theory.
So the outstanding question is: what is some good notion of stability, categorically natural and satisfied in as many of the known models as possible, but strong enough that structures satisfying it can be transformed into ones with \stable{0}{strictly stable} logical structure?

We investigate this within the framework of \myemph{comprehension categories} \cite{jacobs:comprehension-categories} (one of the many algebraic approaches to modelling type theory).
Roughly, a comprehension category consists of a fibration of categories, encoding the contexts, types, and terms of the theory, together with further algebraic structure implementing the logical rules.
Direct models of syntax are given by \myemph{split} comprehension categories with \myemph{\stable{0}{strictly stable}} logical structure.

The coherence construction we consider, transforming weak models into strict, is at the level of underlying fibrations a classical one, due to Giraud \acmarxiv{\citeyear[I,~2.4.3]{giraud:cohomologie-non-abelienne}}{\cite[I,~2.4.3]{giraud:cohomologie-non-abelienne}}.
Given a comprehension category $\CC$, we replace it with an equivalent split one $\CC_{!}$, in which the objects/contexts are as in $\CC$, but where a type $A$ over $\Gamma$ consists of another object $\baseof{A} \in \CC$, together with a type $\topof{A}$ over $\baseof{A}$ in the sense of $\CC$, and a map $\nameof{A} : \Gamma \myto \baseof{A}$.

One may view this as a “delayed substitution”, i.e.\ as a surrogate for the pullback $[A] := \topof{A}[\nameof{A}]$, and see the pair $(\baseof{A},\topof{A})$ as a “local universe”---some family of types, not necessarily terribly large or satisfying any closure conditions, from which the map $\nameof{A} : \Gamma \myto \baseof{A}$ then picks out the types that actually occur in $[A]$.
\begin{align*}
  \begin{tikzpicture}[auto]
    \node (G) at (0,0) {$\Gamma$};
    \node (N) at (1.5,0) {$\baseof{A}.$};
    \node (Ntop) at (1.5,1.5) {$\topof{A}$};
    \draw[->] (G) to node[mylabel] {$\nameof{A}$} (N);
    \draw[dotted] (Ntop) to (N);
  \end{tikzpicture}
\end{align*}

Given a universe $(V,E)$ (not merely a set, but itself an object of $\CC$), closed under operations corresponding to the logical constructors, one can simply restrict to the case where $(\baseof{A},\topof{A})$ is $(V,E)$, so that types are just maps into $V$.
Logical constructions are then modelled by composing these maps with the appropriate operations on $V$.
This is Voevodsky’s approach in \cite[\textsection 1]{kapulkin-lumsdaine-voevodsky:simplicial-model}, also previously considered by Hofmann and Streicher \acmarxiv{\citeyear{hofmann-streicher:lifting-grothendieck-universes}}{\cite{hofmann-streicher:lifting-grothendieck-universes}}.
%

In the absence of such a universe, however, one can implement logical constructions by allowing the local universes to vary.

For instance, suppose we are given types $A$ and $B$, and wish to construct their sum $A + B$.
We do not expect that either of the families $\topof{A} \fibto \baseof{A}$, $\topof{B} \fibto \baseof{B}$ will include the sum types we need.
However, the product $\baseof{A} \times \baseof{B}$ can be used to parametrize all possible sums of a type from $\baseof{A}$ and a type from $\baseof{B}$.
So we take this as the new local universe $\baseof{A+B}$, together with the family of such sums $\topof{A}[\pi_1] + \topof{B}[\pi_2]$ over it (a sum type over $\baseof{A+B}$) as $\topof{A+B}$.
Then for $\nameof{A+B}$ we take $(\nameof{A},\nameof{B}) : \Gamma \myto \baseof{A+B}$, picking out the specific sums required.

Substitution in $\CC_{!}$ corresponds to precomposition with $\nameof{A}$, happening entirely to the left of $\Gamma$; and logical constructions are implemented entirely in terms of the local universes $\topof{A} \fibto \baseof{A}$, all on the right of $\Gamma$.
There is no interference between these operations, so \stable{0}{strict stability} is obtained.

The stability conditions on $\CC$ required to make this work then turn out to be quite weak, and quite categorically natural.
Suppose we are considering a type-constructor, \myemph{widgets}, defined by some universal property.
Then what we need to assume is that $\CC$ has widgets whose pullbacks are again widgets---call such things \myemph{\stable{3}{weakly stable} widgets}.
Given these, widgets in $\CC_!$ may be implemented by choosing \stable{3}{weakly stable} widgets over the local universes, secure in the knowledge that pulled back to the actual contexts involved, they will have the required universal property.
So, typically, if $\CC$ has \stable{3}{weakly stable} widgets, then $\CC_{!}$ will have \stable{0}{strictly stable} widgets.

In case the universal property is strong enough to determine widgets up to isomorphism, \stable{3}{weak stability} is equivalent to a standard Beck-Chevalley condition; but of course, connectives in intensional type theory are often defined by rather weaker properties.

Note also that this stability condition is simply a matter of existence, and does not depend on any specific choices of structure in $\CC$.

One more ambient hypothesis is required, for the manipulation of local universes (for instance, the use of $\baseof{A} \times \baseof{B}$ above): we will require some products and exponentials, in the categorical sense (stronger than the type-theoretic sense), to exist in $\CC$.

Our main theorem is therefore:

\acmarxiv{\begin{theorem}[]}{\begin{theorem*}}
Let $\CC$ be a full comprehension category, whose underlying category has (a) finite products, and (b) exponentials along display maps into display maps and product projections (condition \LF\ below).

Then there is an equivalent full split comprehension category $\CC_!$; and if $\CC$ has \stable{3}{weakly stable} binary sums (resp.\ $\Pi$-types, $\Sigma$-types, identity types, $\W$-types, \ldots), then $\CC_!$ has \stable{0}{strictly stable} binary sums ($\Pi$-types, \ldots), and hence models the syntax of type theory with binary sums ($\Pi$-types, \ldots).
\acmarxiv{\end{theorem}}{\end{theorem*}}

The paper is organized as follows.

First, in Section~\ref{sec:setting}, we survey some background: the general setting of comprehension categories, some existing split replacement constructions, a range of stability conditions for logical structure, and existing results on when such structure lifts to split replacements.

In Section~\ref{sec:main-thm}, we give our main construction, in several stages.
We first set up the ``local universes'' split replacement $\CC_!$, and roughly preview how logical structure will lift to it.
With this as motivation, we set up some technical tools for manipulating universes.
Equipped with these, we are then ready for the full details of the construction.
In Section~\ref{sec:main-proof}, we set out for each logical constructor in turn the precise statement and construction of its lifting to $\CC_!$.
Taken together, these constitute the main theorem.

Finally, in Section~\ref{sec:notes}, we discuss how the construction may be generalized to further logical constructors, and conclude by listing some applications.

\opt{arxiv}{
  \paragraph{Acknowledgements}
  \input{temp-local-universes-acks}
}

\section{Setting} \label{sec:setting}

\subsection{Comprehension categories}

We present models of type theory via the formalism of comprehension categories.
We recall here the key points of this approach; for more detailed background, see \cite{jacobs:comprehension-categories}.

\begin{definition}
  A \myemph{comprehension category} consists of a category $\CC$ together with a (cloven) Grothendieck fibration $P:\TT \myto \CC$ and a functor $\chi:\TT \myto \CC^\arr$ (the \myemph{comprehension}), sending cartesian arrows to pullback squares, and such that
  \begin{align*}
    \begin{tikzpicture}[auto]
      \node (E) at (0,1.5) {$\TT$};
      \node (C) at (1.5,0) {$\CC$};
      \node (C2) at (3,1.5) {$\CC^\arr$};
      \draw[->] (E) to node[mylabel] {$\chi$} (C2);
      \draw[->] (E) to node[mylabel,swap] {$P$} (C);
      \draw[->] (C2) to node[mylabel] {$\cod$} (C);
    \end{tikzpicture}
  \end{align*}
  commutes (strictly).
  We say that a comprehension category is \myemph{split} if $P$ is a split fibration, and \myemph{full} if $\chi$ is full and faithful.

  In the present work (as in much of the literature), all comprehension categories are taken to be full.

  Split comprehension categories are models of an essentially algebraic theory, \acmarxiv{and as such, have}{so carry} an evident notion of homomorphism, forming a category $\SplCompCat$.
  Restricting to some fixed base category $\CC$ (and morphisms acting as the identity on $\CC$) yields a subcategory $\SplCompCat(\CC)$.
  
  Morphisms of non-split comprehension categories are subtler.
  For the present paper, we take such a morphism to consist of a functor $F : \CC' \myto \CC$ of base categories, and a cartesian functor $\bar{F} : \TT' \myto \TT$, strictly over $F$, and commuting strictly with the comprehension functors.
  We write $\CompCat$ for the resulting category; and, again, $\CompCat(\CC)$ for the category of comprehension categories on a fixed base $\CC$.

  (The strict commutation with comprehension is, for many purposes, rather unnatural.
  We take it here just for simplicity, since we make little use of morphisms of general comprehension categories.)
\end{definition}

(For readers more familiar with other algebraic approaches, note that full split comprehension categories are precisely equivalent to categories with attributes, categories with families, type-categories, and the like; contextual categories are also closely comparable to all of these, but not quite equivalent.)

\begin{example}
  The canonical example is given by the syntax of type theory itself.

  Take $\T$ to be any type theory with the judgement forms and structural rules of Martin-L\"of type theory.
  Write $\CC_\T$ for the \myemph{category of contexts} of $\T$: objects are the contexts of $\T$, and arrows are substitutions between contexts, all up to judgemental equality.
  Over this, take $\TT_\T$ to be the category of types-in-context of $\T$, with $p : \TT_\T \myto \CC_\T$ sending a type-in-context to its context; so the fiber $\TT_\T(\Gamma)$ is the category of types over $\Gamma$.
  Reindexing is given by substitution in types; since this is strictly functorial, it makes $p$ into a split fibration.
  Finally, the comprehension operation sends a type-in-context $\judge{\Gamma}{A}$ to the context extension $\Gamma,x\of A$ together with its dependent projection $\pi_{\Gamma,A} : (\Gamma,x \of A) \myto \Gamma$.

  Together, these form a (full) split comprehension category $\CC_\T$.
\end{example}

This example motivates much of the terminology and notation for general comprehension categories $(\CC,\TT,p,\chi)$.

The objects of the base category $\CC$ are thought of as \myemph{contexts}; and given such an object $\Gamma$, we consider objects of the fiber $\TT(\Gamma)$ as \myemph{types over $\Gamma$}.

Given a type $A \in \TT(\Gamma)$, its comprehension $\chi(A)$ is an arrow with codomain $\Gamma$.
We denote the domain of $\chi(A)$ by $\Gamma.A$; it may be seen as the context extension of $\Gamma$ by some new variable of type $A$, with $\chi(A)$ the resulting dependent projection:
\begin{align*}
  \begin{tikzpicture}[auto]
      \node (GA) at (0,1.25) {$\Gamma.A$};
      \node (G) at (0,0) {$\Gamma.$};
      \draw[->] (GA) to node[mylabel] {$\chi(A)$} (G);
  \end{tikzpicture}
\end{align*}
We refer to composites of such maps as \myemph{display maps}.

Next, given a map $\sigma:\Delta \myto \Gamma$ (which we think of as a context morphism, or substitution) and an object $A$ of $\TT(\Gamma)$, we write $\sigma_{A}:A[\sigma] \myto \Gamma$ for the lift provided by the cleaving of the fibration $P$.
Diagrammatically:
\begin{align*}
  \begin{tikzpicture}[auto]
    \node (Delta) at (0,0) {$\Delta$};
    \node (Gamma) at (3,0) {$\Gamma$};
    \node (A) at (3,2) {$A$};
    \node (As) at (0,2) {$A[\sigma]$};
    \node (E) at (4,2) {$\TT$};
    \node (C) at (4,0) {$\CC.$};
    \draw[->] (Delta) to node[mylabel] {$\sigma$} (Gamma);
    \draw[->] (E) to node[mylabel] {$P$} (C);
    \draw[->] (As) to node[mylabel] {$\sigma_{A}$} (A);
    \draw[dotted] (A) to (Gamma);
    \draw[dotted] (As) to (Delta);
  \end{tikzpicture}
\end{align*}
The type $A[\sigma]$ may be seen as the result of applying the substitution $\sigma$ to $A$.
Following our notation for context extensions, we denote the comprehension $\chi(\sigma_A)$ by $\sigma.A : \Delta.A[\sigma] \myto \Gamma.A$.

Given our standing assumption of fullness, we will also sometimes silently conflate maps in fibers of $\TT$ with maps in slices of $\CC$.

In syntactic categories $\CC_\T$, terms of a type $A$ in context $\Gamma$ correspond to sections $t : \Gamma \myto \Gamma.A$ of the dependent projection $\Gamma.A$; such sections occur frequently when working with comprehension categories.
We will therefore often write just “a section” to mean “a section of a dependent projection”, unless specified otherwise.
\begin{align*}
  \begin{tikzpicture}[auto]
    \node (GA) at (0,2) {$\Gamma.A$};
    \node (Gtop) at (-2,1) {$\Gamma$};
    \node (Gbottom) at (0,0) {$\Gamma$};
    \draw[->] (Gtop) to node[mylabel] {$a$} (GA);
    \draw[->,bend right=15] (Gtop) to node[mylabel,swap] {$1_{\Gamma}$} (Gbottom);
    \draw[->] (GA) to node[mylabel] {$\chi(A)$} (Gbottom);
  \end{tikzpicture}
\end{align*}

Finally, for a map $\sigma : \Delta \myto \Gamma$, we extend the reindexing notation $A[\sigma]$ in several ways.
Most straightforwardly, it denotes the reindexing functor $\TT(\Gamma) \myto \TT(\Delta)$, with action on objects provided by the cleaving of $\TT$.
More generally, given a map $f : A \myto B$ in $\TT(\Gamma)$, and \myemph{any} specified reindexings $A',B'$ of $A$ and $B$ (precisely, cartesian lifts $\sigma_A : A' \myto A$, $\sigma_B : B' \myto B$ of $\sigma$), we write $f[\sigma]$ for the induced map $A' \to B'$ in $\TT(\Delta)$.
Finally, we write $(-)[\sigma]$ also for the induced pullback functor $\CC \fibslice \Gamma \myto \CC \fibslice \Delta$, where $\CC \fibslice \Gamma$ denotes the full subcategory of $\CC/\Gamma$ with just display maps as objects.

All the above notation and terminology suggests that the syntactic comprehension categories $\CC_\T$ may be seen as typical.
This is indeed the case; specifically, they turn out to enjoy certain universal properties, arguably the \emph{raison d'\^e{}tre} of comprehension categories.

Starting with $\T_{\emptyset}$, the theory given by just the structural rules, one has:
\begin{proposition}[\acmarxiv
    {\cite[\textsection 15]{cartmell:generalised-algebraic-theories}}
    {Cartmell \cite[\textsection 15]{cartmell:generalised-algebraic-theories}}]
  $\CC_{\T_{\emptyset}}$ is the initial split comprehension category.
\end{proposition}

In other words, any split comprehension category $\CC$ admits a canonical map $\interp{-} : \CC_{\T_\emptyset} \myto \CC$; that is, an interpretation of the syntax of $\T_{\emptyset}$ in $\CC$.
This justifies taking split comprehension categories as a definition of \myemph{models of $\T_{\emptyset}$}.

Extending this correspondence to less trivial type theories, one considers comprehension categories with extra structure.
Take, for instance, the theory $\T_\tensor$ given by the structural rules together with a single type-forming rule:
\begin{mathpar}
  \inferrule*{\Gamma \types A \type \\ \Gamma \types B \type}{\Gamma \types A \tensor B \type}
\end{mathpar}

Say that \myemph{(\stable{0}{strictly stable}) $\tensor$-structure} on a comprehension category $\CC$ consists of an operation giving, for any objects $\Gamma \in \CC$ and $A,B \in \TT(\Gamma)$, an object $A \tensor_\Gamma B \in \TT(\Gamma)$, \stable{0}{strictly stable} under reindexing, i.e.\ such that for any such $\Gamma, A, B$ and map $\sigma : \Gamma' \myto \Gamma$, we have $(A \tensor_\Gamma B)[\sigma] = A[\sigma] \tensor_{\Gamma'} B[\sigma]$.
Then $\CC_{\T_\tensor}$ carries an evident $\tensor$-structure (with its \stable{0}{strict stability} arising inevitably from the inductive definition of substitution, $(A \tensor B)[\sigma] := A[\sigma] \tensor B[\sigma]$); and indeed, we have:
\begin{proposition}
The syntactic category $\CC_{\T_\tensor}$ is initial in the category of split comprehension categories with $\tensor$-structure, and functors strictly preserving the comprehension functor and $\tensor$-structure.
\end{proposition}

In other words, any split comprehension category $\CC$ with \stable{0}{strictly stable} $\tensor$-structure carries a canonical interpretation $\interp{-} : \CC_{\T_\tensor} \myto \CC$ of the syntax of $\T_\tensor$.
This justifies once again taking such comprehension categories as an algebraic definition of \myemph{models of $\T_\tensor$}.

Similarly, this correspondence extends to all the other usual constructors and rules.
Each new constructor corresponds to an extra operation, and each new judgemental equality rule to an algebraic axiom.
The syntactic category $\CC_\T$ is then initial among split comprehension categories with the appropriate \stable{0}{strictly stable} structure.
See \cite[Ch.~3, p.~181]{streicher:semantics-book} for a detailed treatment of the case of the Calculus of Constructions.

The task of modelling type theory thus amounts to the construction of split comprehension categories with suitable \stable{0}{strictly stable} algebraic structure.

Unfortunately, many constructions of models do not directly give this: they give comprehension categories with the appropriate algebraic structure, but they are not split, and the operations are not \stable{0}{strictly stable}.
This occurs particularly in abstract categorical constructions, such as the model of extensional type theory (ETT) in an arbitrary LCCC \cite{seely:lcccs,hofmann:lcccs}, or of intensional type theory (ITT) using a suitable weak factorization system \cite{warren-awodey:homotopy-theoretic-models}.
In such models, $\TT(\Gamma)$ usually consists of all maps into $\Gamma$ satisfying some property, or carrying some extra structure.
Reindexing is then given by pullback; but this is defined only up to canonical isomorphism, and so is only functorial up to isomorphism, not on the nose.
By the same token, the logical operations are typically characterized at most up to canonical isomorphism (and in some homotopy-theoretic models, only up to homotopy equivalence), and so are not automatically \stable{0}{strictly stable}.

In concrete models, a solution can often be found: an equivalent comprehension category, split, and with some \stable{0}{strictly stable} choice of structure.
For the model in $\sets$, for instance, one takes $\TT(\Gamma)$ not as $\sets/\Gamma$, but as the equivalent category $\sets^\Gamma$.
Related solutions exist for the (higher) groupoid models \cite{hofmann-streicher:groupoid-model,warren:thesis,warren:strict-w-groupoid-model}, presheaf models of ETT, and various other concrete examples.

However, for general categorical constructions, and \acmarxiv{for}{} some homotopy-theoretic models (e.g.\ $\ssets$), the problem remains: when can one replace a comprehension category, carrying some kind of logical structure, by an equivalent split one with \stable{0}{strictly stable} structure?
This is the \myemph{coherence problem} for dependent type theory.

In the remainder of this section, we survey existing results and lay out the setting for our own, taking the problem in two steps: first the splitness, then the \stable{0}{strict stability}.

\subsection{Split replacements}

The split replacement constructions for fibrations are classical, due to Giraud \acmarxiv{\citeyear[I, 2.4.3]{giraud:cohomologie-non-abelienne}}{\cite[I, 2.4.3]{giraud:cohomologie-non-abelienne}} and Bénabou \cite{streicher:benabou-splitting-email}.
Let $\Fib(\CC)$ denote the (1-)category of (cloven) fibrations and cartesian functors (not assumed to preserve the cleaving) over a fixed base $\CC$, and similarly $\SplFib_\clov(\CC)$ the category of split fibrations over $\CC$, with morphisms \myemph{cloven} functors, i.e.\ preserving the splitting on the nose.
Then the inclusion functor $i : \SplFib_\clov(\CC)\myto \Fib(\CC)$ possesses both left and right adjoints:
\begin{align*}
  \begin{tikzpicture}[auto]
    \node (Split) at (0,2) {$\SplFib_\clov(\CC)$};
    \node (Fib) at (0,0) {$\Fib(\CC)$.};
    \draw[->,bend right=40] (Fib) to node[mylabel,swap] {$(-)_{*}$}
    (Split);
    \draw[->,bend left=40] (Fib) to node[mylabel] {$(-)_{!}$} (Split);
    \draw[->] (Split) to node {$\adjoint$} node[swap] {$\adjoint$} (Fib);
  \end{tikzpicture}
\end{align*}
Explicitly, let $p : \TT \myto \CC$ be a cloven fibration.
An object of $\TT_*$ over $\Gamma \in \CC$ consists of an object $A$ of $\TT(\Gamma)$, together with for each $f : \Gamma' \myto \Gamma$ some cartesian lifting $\overline{f} : A_f \myto A$, such that $A_{1_\Gamma} = A$, $\overline{1_\Gamma} = 1_A$.
On the other hand, an object of $\TT_!$ over $\Gamma$ consists of objects $V \in \CC$, $E \in \TT(V)$, and a map $f : \Gamma \myto V$ (as discussed in detail in Section~\ref{sec:c-bang} below).

The unit and counit maps of these adjunctions are not in general isomorphisms.
However, considering $\Fib(\CC)$ as a 2-category (with 2-cells natural transformations over $\CC$), they are equivalences.

In other words, every fibration over a given base has two canonical split replacements. \\

Moreover, much of this situation lifts from fibrations to comprehension categories.

For the right adjoint $(-)_*$, this is described in \cite{hofmann:lcccs,curien-garner-hofmann}.
Given a comprehension category $\CC = (\CC,\TT,p,\chi)$, the comprehension $\chi_*$ on $\TT_*$ is given by sending an object-with-chosen-reindexings simply to the comprehension in $\CC$ of the object itself.

Similarly, the left adjoint lifts to a functor 
\[ (-)_! : \CompCat(\CC) \myto \SplCompCat(\CC), \]
described in full in Section~\ref{sec:c-bang} below, with the comprehension on $\TT_!$ sendings an object $f:\Gamma \myto V$, $A \in \TT(V)$ to $\chi(A[f])$.
This is no longer a left adjoint, however: the putative unit map $(\CC,\TT,p,\chi) \myto (\CC,\TT_!,p_!,\chi_!)$ will strictly preserve the comprehension just when $p$ is a \myemph{normal} fibration, i.e.\ when the cleaving lifts identity maps to identities.

We retain, however, the fact that the maps $\CC_* \myto \CC$ and $\CC \myto \CC_!$ are equivalences in suitable 2-categories.
%

In sum, we find that simply for bare comprehension categories, the coherence problem is satisfactorily solved: $(-)_*$ and $(-)_!$ provide two ways to replace an arbitrary comprehension category by an equivalent split one.

\subsection{Stability conditions} \label{sec:stability-conditions}

The real fun starts when one wishes to model a non-trivial type theory; that is, when one has some logical structure on the original comprehension category, and wishes to lift it to \stable{0}{strictly stable} structure on a split replacement.

It is reasonable to expect that \myemph{some} kind of stability condition will be needed for the operations of the original category.
We set out here a range of possible such conditions, from stronger ones which will lift more easily, to weaker ones which are satisfied more often in nature.

We do not give general definitions of them, for arbitrary logical operations: no appropriate generality of operations exists in the literature, and giving one is beyond the scope of this paper.
Instead, we define them here in the illustrative case of identity types; then in Section~\ref{sec:main-proof} below, we define them for other specific type-constructors as we require them.

Fix, for the remainder of this section, a comprehension category $\CC = (\CC,\TT,p,\chi)$.

\begin{definition} \label{def:weak-id-types}
  Given objects $\Gamma \in \CC$ and $A \in \TT(\Gamma)$, an \myemph{identity type for $A$} consists of:
  \begin{itemize}
  \item a type $\id[A] \in \TT(\Gamma.A.A)$\footnote{Strictly, $\TT(\Gamma.A.A[\chi(A)])$; here and elsewhere, we suppress such “weakenings”, i.e.\ reindexings along dependent projections, where they are unambiguous.};
  \item a factorization
    \begin{align*}
      \begin{tikzpicture}[auto]
        \node (L) at (0,1.5) {$\Gamma.A$};
        \node (R) at (3,1.5) {$\Gamma.A.A.\id[A]$};
        \node (M) at (1.5,0) {$\Gamma.A.A$};
        \draw[->] (L) to node[mylabel] {$\rr_{A}$} (R);
        \draw[->,bend right=15] (L) to node[mylabel,swap] {$\Delta$} (M);
        \draw[->,bend left=15] (R) to (M);
      \end{tikzpicture}
    \end{align*}
    of the diagonal map $\Delta_A : \Gamma.A \myto \Gamma.A.A$;
  \item for each $C\in\TT(\Gamma.A.A.\id[A])$ and $d:\Gamma.A\myto \Gamma.A.A.\id[A].C$ such that the outer square of
  \begin{align*}
    \begin{tikzpicture}[auto]
      \node (UL) at (0,1.75) {$\Gamma.A$};
      \node (BL) at (0,0) {$\Gamma.A.A.\id[A]$};
      \node (BR) at (3.3,0) {$\Gamma.A.A.\id[A]$};
      \node (UR) at (3.3,1.75) {$\Gamma.A.A.\id[A].C$};
      \draw[->] (UL) to node[mylabel] {$d$} (UR);
      \draw[->] (UL) to node[mylabel,swap] {$\rr_{A}$} (BL);
      \draw[->] (BL) to node[mylabel,swap] {$1_{\Gamma.A.A.\id[A]}$} (BR);
      \draw[->] (UR) to (BR);
      \draw[->,dashed] (BL) to node[mymidlabel] {$\jj_{A,C,d}$} (UR);
    \end{tikzpicture}
  \end{align*}
  commutes, a diagonal filler $\jj_{A,C,d}:\Gamma.A.A.\id[A]\myto C$ as indicated, making the resulting triangles commute.
  \end{itemize}

  A \myemph{choice of identity types on $\CC$} is a function giving, for each appropriate $\Gamma$, $A$ in $\CC$, an identity type $(\id[A],\rr_A,\jj_A)$ for $A$.
\end{definition}

To give a direct model of syntactic identity types, one requires stability conditions corresponding to the recursive definition of syntactic substitution:
\begin{definition}
  A choice of identity types on $\CC$ is \myemph{\stable{0}{strictly stable}} if for each $\sigma : \Gamma' \myto \Gamma$, and all appropriate $A$, $C$, $d$,
  \begin{align*}
    \id[A][\sigma] & = \id[{A[\sigma]}]\\
    \rr_{A}[\sigma] & = \rr_{A[\sigma]}\\
    \jj_{A,C,d}[\sigma] & = \jj_{A[\sigma],C[\sigma],d[\sigma]}.
  \end{align*}
\end{definition}

The most problematic aspect of this definition, categorically, is that it requires an on-the-nose equality of types.
Also, it does not necessarily respect isomorphism of types.
Making the minimal modification to allay these objections, we obtain:
\begin{definition}
  A choice of identity types on $\CC$ is \myemph{\stable{1}{(fully) pseudo-stable}} if it is equipped with a cartesian functorial action on cartesian maps.  That is, for each $\sigma : \Gamma' \myto \Gamma$, and cartesian map $\sigma_A : A' \myto A$ over $\sigma$, a cartesian map
  \[ \id[\sigma_A] : \id[A'] \myto \id[A] \]
over $\sigma.\sigma_A.\sigma_A : \Gamma'.A'.A' \myto \Gamma.A.A$, such that $\id[1_A] = 1_{\id[A]}$, $\id[\tau_A \comp \sigma_A] = \id[\tau_A] \comp \id[\sigma_A]$, and moreover commuting appropriately with values of $\rr$ and $\jj$.

(Full details of the ``commuting appropriately'' are spelled out in \cite[Def.\ 2.38]{warren:thesis}, where these are called \myemph{coherent identity types}; we omit them here, since our main results do not involve this condition.)
\end{definition}

This action may be seen as combining two parts: comparison isomorphisms $\id[{A[\sigma]}] \iso \id[A][\sigma]$ for the reindexings given by the cleaving, and functoriality in isomorphisms $A \iso B$ within each fiber $\TT(\Gamma)$.
An earlier version of this paper incorrectly defined pseudo-stability using just the first of these components.

In homotopy-theoretic models, the fillers $\jj$ are usually not specified, but given merely by an existence condition.
This suggests the further weakening:
\begin{definition}\label{def:stableid}
  A choice of partly-specified identity types (i.e.\ operations giving chosen $\id[A]$, $\rr_A$, such that there exist elimination fillers $\jj$ making these identity types), is \myemph{\stable{2}{partially pseudo-stable}} if it is equipped with a cartesian action of $\id$ on cartesian maps (as above), commuting with values $\rr$ (but not necessarily $\jj$).

(Again, see \cite[Def.\ 2.33]{warren:thesis} for full details; these are the \myemph{stable identity types} there.)
\end{definition}

A more categorically familiar property is the \myemph{Beck--Chevalley} condition:
\begin{definition}\label{def:beck-chevalley}
  Say a choice of identity types on $\CC$ satisfies the \myemph{Beck--Chevalley condition} if for each $\sigma : \Gamma' \myto \Gamma$ and $A \in \TT(\Gamma)$, the canonical map
  \[ \jj_{A[\sigma],\id[A][\sigma],\rr_A[\sigma]} : \id[{A[\sigma]}] \myto \id[A][\sigma] \]
  is an isomorphism.
  This depends on these particular fillers, but not on other values of $\jj$.
\end{definition}

Our final condition is a pure existence condition on $\CC$, not dependent at all on a choice of identity types:
\begin{definition}\label{def:weakly-stable}
  Given $\Gamma$ in $\CC$ and $A$ in $\TT(\Gamma)$, a \myemph{\stable{3}{weakly stable} identity type} for $A$ is a pair $(\id,\rr)$ as above such that, for all $\sigma:\Gamma' \myto \Gamma$, there is some $\jj$ making $(\id[A][\sigma],\rr_{A}[\sigma],\jj)$ an identity type for $A[\sigma]$.

  Say that $\CC$ \myemph{has \stable{3}{weakly stable} identity types} if for every $\Gamma$, $A$, there is some \stable{3}{weakly stable} identity type $(\id,\rr)$.
\end{definition}

The conditions above may be usefully compared in terms of their impliciations for the Beck--Chevalley maps $ \jj_{A[\sigma],\id[A][\sigma],\rr_A[\sigma]} : \id[{A[\sigma]}] \myto \id[A][\sigma]$, and for the stability of values of $\jj$ modulo these maps.

\begin{center}
  \begin{tabular}{l|cc@{\quad~\quad}c}
    choice of $\id$-types & $(\id,\rr)$ & $\jj$\\
    \hline
    \emph{\stable{0}{strictly stable}} & $=$        & $=$ \\
    \emph{\stable{1}{pseudo-stable}}   & $\iso$     & $=$ \\
    \emph{\stable{2}{partially pseudo-stable}}     & $\iso$     & $\homot$ \\
    \emph{\stable{3}{weakly stable}}     & $\myequiv$ & $\homot$ \\
    \emph{arbitrary}         &            &  \\
  \end{tabular}
\end{center}
Here $\homot$ denotes homotopy (pointwise propositional equality), and $\myequiv$ homotopy equivalence, in the sense of \cite[Ch.~4]{hott:book}.

Analogous definitions of these levels of \stable{other}{stability} may be made for the other usual type and term constructors; as mentioned already, we will define these in full as they are required, in Sec.~\ref{sec:main-proof} below.
Briefly, they are obtained for inductive types ($+$, $\Sigma$, $1$, etc.)\ by replacing $(\id,\rr)$ in the definitions above by the type former in question and its constructors, and replacing $\jj$ by the eliminator; and for $\Pi$-types, by replacing $(\id,\rr)$ by the $\Pi$-type and its application map, and $\jj$ by the $\lambda$-abstraction operation. \\

\paragraph{Existence conditions and the axiom of choice}

The various existence conditions above---in particular, \stable{3}{weak stability}---may each be interpreted in two ways: classically, as mere existence, or according to the constructive tradition, with each forall--exists statement witnessed by some function (but with no conditions on the function assumed).

Assuming the axiom of choice, the two are of course equivalent.
In the absence of AC, however, the witnessed form is stronger, and is the form required for the results of Section~\ref{sec:main-thm} below.
(Compare the use of \myemph{cloven} fibrations in the definition of comprehension categories.)
We will for the most part elide this distinction; where necessary, we will speak of \myemph{witnessed \stable{3}{weakly stable} identity types}, and the like.

\subsection{Lifting logical structure}

Equipped with these definitions, we can now state when logical structure lifts from $\CC$ to its strict replacements $\CC_*$, $\CC_!$.

For the right-handed strictification $\CC_*$, the known results require either restrictions on the type theory in question, or strong stability conditions.

\begin{theorem}[\acmarxiv
  {{\cite[Thm.\ 2,4]{hofmann:lcccs}}, \cite{curien-garner-hofmann}}
  {Hofmann {\cite[Thm.\ 2,4]{hofmann:lcccs}}, analysed further in \cite{curien-garner-hofmann}}]
  \label{thm:hofmann-c-star}
Suppose $\CC$ is a comprehension category, equipped with structure corresponding to the logical constructions of \myemph{extensional} Martin-L\"of type theory, including in particular the reflection rule for identity types\footnote{That is, for each $A \in \TT(\Gamma)$, there is a type $\id[A]$ over $\Gamma.A.A$ whose comprehension is isomorphic over $\Gamma.A.A$ to the diagonal map $\Gamma.A \myto \Gamma.A.A$.}, and all satisfying the appropriate Beck--Chevalley conditions.

Then $\CC_*$ may be equipped with \stable{0}{strictly stable} $\Pi$-structure, $\Sigma$-structure, etc.

In particular, if $\EE$ is a locally cartesian closed category, then $(\EE,\EE^\arr,\cod,1)_*$ is a model of extensional type theory.
\end{theorem}

With the terminology above, this may be read as factoring into two lemmas.
The first characterizes when logical structure (not just that corresponding to ETT) lifts to $\CC_*$.

\begin{lemma} \label{lem:lifting-to-c-star}
Suppose $\CC$ is a comprehension category, equipped with \stable{1}{pseudo-stable} $\id$-types (resp.\ $\Pi$-types, $+$-types, etc.).

Then $\CC_*$ carries \stable{0}{strictly stable} $\id$-types ($\Pi$-types, $+$-types, etc.).
\end{lemma}

\begin{proof}
Just as in the proof of \cite[Thm.\ 2]{hofmann:lcccs}
\end{proof}

The second shows why for extensional type theory, and similar theories, \stable{1}{pseudo-stability} is a very reasonable condition to expect.

\begin{lemma}
Suppose $\CC$ is a comprehension category, with identity types satisfying the reflection rule, and with $\Pi$-types (resp.\ $+$-types, etc.), all \stable{3}{weakly stable} (or, \emph{a fortiori}, satisfying the Beck--Chevalley condition), and in the case of $\Pi$-types, satisfying the $\eta$-rule.

Then all this structure is in fact \stable{1}{pseudo-stable}, with comparison isomorphisms corresponding to the Beck--Chevalley maps.
\end{lemma}

\begin{proof}Using the identity types with the reflection rule, one may show that the maps produced by eliminators of inductive types (e.g.\ $\jj$), or by $\lambda$-abstraction for $\Pi$-types with $\eta$, are unique; e.g.\ given identity-elimination data $\Gamma$, $A$, $C$, $d$, and an $\id$-type $(\id[A],\rr_A,\jj_A)$  for $A$, then any section $f$ of $\chi(C)$ with $f \comp \rr_A = d$ must be equal to $\jj_{A,C,d}$.
From this, it follows that the various type formers have categorical universal properties, either as certain initial algebras or as exponential objects.
\stable{3}{Weak stability} then implies that the Beck--Chevalley maps, as algebra maps between two initial/terminal objects, are isomorphisms, and satisfy the appropriate axioms to witness that the logical structure is \stable{1}{pseudo-stable}.
\end{proof}

For theories such as intensional Martin-L\"of type theory, however, \stable{1}{pseudo-stability} is often difficult to obtain.

Most obviously, type constructors such as $\id$-types, or sum types without $\eta$-rules, are not automatically determined up to isomorphism, only up to equivalence.
This is not often an obstacle in practice, though, since the specific constructions used usually are stable up to coherent isomorphism after all.

More problematically, however, the $\id$-elimination operation (and some other inductive eliminators) is not canonically determined, but merely given by existence conditions, so does not commute with substitution and the coherence isomorphisms; and even in concrete cases (e.g.\ $\ssets$), it may not be clear how to make choices that do \cite[\textsection 2.3]{warren:thesis}.

In the terminology from above, identity types in comprehension categories coming from homotopy-theoretic models are usually \stable{2}{partially pseudo-stable}, but are often not \stable{1}{fully so}.
It is not possible, therefore, to apply Lemma~\ref{lem:lifting-to-c-star} to obtain strictifications of such models.

Theorem~\ref{thm:main-thm} and Heuristic~\ref{heuristic:further-rules} below resolve this situation, stating that for lifting logical structure to the \myemph{left-handed} strictification $\CC_!$, only \stable{3}{weak stability} is required, provided that certain products and exponentials exist in the base category $\CC$.

The next section is devoted to the full statement and proof of this theorem.

\section{The local universes model} \label{sec:main-thm}

\subsection{The comprehension category \pdfCbang} \label{sec:c-bang}

Throughout this section, assume given a fixed comprehension category $\CC = (\CC,\TT,p,\chi)$.

\begin{definition}
  Define the new comprehension category $\CC_! = (\CC_!,\TT_!,p,\chi_!)$ as follows:
  \begin{description}
  \item[Base category] We set $\CC_{!} := \CC$; the base category does not change.
  \item[Types] An object of $\TT_{!}$ over $\Gamma \in \CC$ consists of a tuple $(\baseof{A},\topof{A},\nameof{A})$, where $\baseof{A}$ is an object of $\CC$, $\topof{A}$ an object of $\TT(\baseof{A})$, and $\nameof{A}$ an arrow $\Gamma \myto \baseof{A}$ in $\CC$.
    One may view this diagrammatically as follows:
    \begin{align*}
      \begin{tikzpicture}[auto]
        \node (G) at (0,0) {$\Gamma$};
        \node (base) at (2,0) {$\baseof{A}.$};
        \node (top) at (2,1.5) {$\topof{A}$};
        \draw[->] (G) to node[mylabel] {$\nameof{A}$} (base);
        \draw[dotted] (top) to (base);
      \end{tikzpicture}
    \end{align*}
    Given such an object, write $[A]$ for its reindexing $\topof{A}[\nameof{A}]$ in $\TT(\Gamma)$.

    An arrow $(\baseof{B},\topof{B},\nameof{B}) \myto (\baseof{A},\topof{A},\nameof{A})$ in $\TT_{!}$ over $\sigma : \Delta \myto \Gamma$ is just a map $[B] \myto[] [A]$ over $\sigma$ in $\TT$:
    \begin{align*}
      \begin{tikzpicture}[auto]
        \node (B) at (0,1.5) {$[B]$};
        \node (A) at (2,1.5) {$[A]$};
        \node (D) at (0,0) {$\Delta$};
        \node (G) at (2,0) {$\Gamma$};
        \draw[->] (D) to node[mylabel] {$\sigma$} (G);
        \draw[->] (B) to (A);
        \draw[dotted] (B) to (D);
        \draw[dotted] (A) to (G);
      \end{tikzpicture}
    \end{align*}

    Together, this gives the category $\TT_!$ together with a projection $p_! : \TT_! \myto \CC_!$.

  \item[Reindexing]
    Cartesian lifts for $p_!$ are given as follows.
    Let $\sigma:\Delta \myto \Gamma$ be an arrow in $\CC$, and $A = (\baseof{A},\topof{A},\nameof{A})$ an object of $\TT_{!}(\Gamma)$.
    Set
    \begin{align*}
      A[\sigma] & := (\baseof{A},\topof{A},\nameof{A} \comp \sigma)
    \end{align*}
    and take the map $A_{\sigma} : A[\sigma] \myto A$ over $\sigma$ to be the canonical map $[A[\sigma]] \myto[] [A]$ over $\sigma$ in $\TT$ given by the cartesianness of $(\topof{A})_{\nameof{A}}$ for $p$:
    \begin{align*}
      \begin{tikzpicture}[auto]
        \node (As) at (0,1.5) {$[A[\sigma]]$};
        \node (D) at (0,0) {$\Delta$};
        \node (A) at (2,.75) {$[A]$};
        \node (G) at (2,-.75) {$\Gamma$};
        \node (Nt) at (6,1.5) {$\topof{A}$};
        \node (N) at (6,0) {$\baseof{A}$};
        \draw[->,bend left=10] (As) to node[mylabel]
        {$(\topof{A})_{\nameof{A}\comp f}$} (Nt);
        \draw[->] (D) to (N);
        \draw[->,bend right=5] (A) to node[swap,mylabel]
        {$(\topof{A})_{\nameof{A}}$} (Nt);
        \draw[->,bend right=5] (G) to node[mylabel,swap]
        {$\nameof{A}$} (N);
        \draw[->] (D) to node[mylabel,swap] {$\sigma$} (G);
        \draw[->,dashed] (As) to node[mylabel,swap] {$A_{\sigma}$}
        (A);
        \draw[dotted] (As) to (D);
        \draw[dotted,preaction={draw=white, -,line width=6pt}] (A) to (G);
        \draw[dotted] (Nt) to (N);
      \end{tikzpicture}
    \end{align*}
    This is straightforwardly seen to make $p_{!}$ a \myemph{split} fibration.
  \item[Comprehension]
    Given an object $(\baseof{A},\topof{A},\nameof{A})$ of $\TT_{!}(\Gamma)$, take 
    \begin{align*}
      \chi_{!}(\baseof{A},\topof{A},\nameof{A}) & := \chi([A]).
    \end{align*}
    \noindent Cartesianness of the functor $\chi_!$ follows directly from that of $\chi$.
  \end{description}
\end{definition}

Intuitively, we think of $\baseof{A}$ together with $\topof{A}$ as a kind of ``local universe''.
(By abuse of notation, we often refer to the pair $(\baseof{A},\topof{A})$ just as $\baseof{A}$, leaving $\topof{A}$ understood.)
Following this intuition, the map $\nameof{A}:\Gamma \myto \baseof{A}$ picks out the actual type family $[A]$ from the local universe $\baseof{A}$. \\

In general, $\CC_{!}$ may not support the interpretation of any interesting constructors, even when $\CC$ does.
However, provided that the underlying category $\CC$ comes equipped with a modest amount of additional structure it will be possible to lift the interpretations of various constructors from $\CC$ to $\CC_{!}$.
We first recall some definitions.

\begin{definition} \label{def:cat-exponential}
If $Z \myto[g] Y \myto[f] X$ are maps in a category $\CC$ such that all pullbacks of $f$ exist, a \myemph{(categorical, dependent) exponential} for $f$ and $g$ is an object $\prod[f,g]$ of $\CC/X$ together with a natural isomorphism 
\[\CC/X(W, \prod[f,g]) \iso \CC/Y(W \times_X Y, Z),\]
for all $h : W \myto X$.

A map $f : Y \myto X$ possessing all pullbacks is \myemph{(dependently) exponentiable} if for every $g : Z \myto Y$, some dependent product $\prod[f,g]$ exists; equivalently, if the pullback functor $f^* : \CC/X \myto \CC/Y$ has a right adjoint.
\end{definition}

Then the precise ambient hypothesis required, for lifting structure to $\CC_!$, is as follows (named by analogy with the logical framework presentation):
\begin{definition} \label{def:LF}
Say that $\CC$ satisfies \myemph{condition \LF} if its underlying category has finite products, and given maps $Z \myto[g] Y \myto[f] X$, if $f$ is a display map and $g$ is either a display map or a product projection, then a dependent exponential $\prod[f,g]$ exists.
\end{definition}

\begin{proposition}
Each of the following (simpler) conditions implies \LF:
\acmarxiv{\begin{describe}{\emph{LCCC}}}{\begin{description}}
\item[\emph{(LFa)}] \label{cond:lfa} $\CC$ has finite products; and every display map is exponentiable.
\item[\emph{(LFb)}] \label{cond:lfb} Every map $X \myto 1$ is a display map; and for each $A \in \TT(\Gamma)$, the reindexing functor $\chi(A)^* : \TT(\Gamma) \myto \TT(\Gamma.A)$ has a right adjoint.
\item[\emph{(LCCC)}] $\CC$ is locally Cartesian closed. \qed
\acmarxiv{\end{describe}}{\end{description}}
\end{proposition}

The exponentials required by \LF\ can be essentially independent of any function types one may consider in the type theory.
On the one hand, they are not required to themselves be display maps.
On the other, they \myemph{are} required to be categorical exponentials, not merely type-theoretic function types (in general slightly weaker).
Compare how in logical framework presentations of the type theory, one usually asks for strong function types in the ambient logical framework, independently of what function types the object theory may possess \cite[Ch.~19]{nordstrom-petersson-smith}.

Given this assumption, all standard type-constructors will lift from $\CC$ to $\CC_{!}$ essentially independently.
We consider them one by one in Sections~\ref{sec:first-strux-sec}--\ref{sec:last-strux-sec}, which all have roughly the same form: we define precisely what it means for a comprehension category to have \stable{3}{weakly stable} widgets, and to have a \stable{0}{strictly stable} choice of widgets; and we show that if $\CC$ has \stable{3}{weakly stable} widgets and satisfies \LF, then $\CC_{!}$ has a \stable{0}{strictly stable} choice of widgets, and hence models type theory with widgets.
Before embarking on this, however, we first set out the general template for the construction of the structure on $\CC_!$, and set up some machinery that it requires.

We assume throughout the following discussion that $\CC$ satisfies condition \LF; however, we explicitly re-state this hypothesis in all theorems, as required.

\subsection{Template for structure on \pdfCbang} \label{sec:template} 

To illustrate the pattern we will follow, first consider the operation corresponding to the $(+)$-formation rule:
\begin{mathpar}
  \inferrule*{\Gamma \types A \type \\ \Gamma \types B \type}{\Gamma \types A + B \type}
\end{mathpar}
Assuming that $\CC$ is equipped with such an operation, we wish to lift it to $\CC_!$.

The result should send any object $\Gamma \in \CC_!$ and pair of types $A_1,A_2 \in \TT_!(\Gamma)$ to some type $A_1 + A_2 \in \TT_!(\Gamma)$.
These inputs correspond, in $\CC$, to objects $\Gamma$, $\baseof{A_i} \in \CC$, types $\topof{A_i} \in \TT(\baseof{A_i})$, and maps $\nameof{A_i} : \Gamma \myto \baseof{A_i}$.

One cannot directly take the sum $\topof{A_1} + \topof{A_2}$ in $\CC$, since these live in different fibers of $\TT$.
One could pull both $\topof{A_1}$, $\topof{A_2}$ back to $\TT(\Gamma)$ and take their sum there, but since this involves $\Gamma$, the resulting operation would not be \stable{0}{strictly stable} unless $(+)$ were already so in $\CC$.

Instead, note that taken together, the maps $\nameof{A_i}$ correspond to the single map $(\nameof{A_1},\nameof{A_2}) : \Gamma \myto \baseof{A_1} \times \baseof{A_2}$, and factor through it via the projections $\pi_1$, $\pi_2$.
We can thus pull both $\topof{A_1}$, $\topof{A_2}$ back to $\TT(\baseof{A_1} \times \baseof{A_2})$, and take their sum there.
Putting this together, we set $\baseof{A_1 + A_2} := \baseof{A_1} \times \baseof{A_2}$, $\topof{A_1 + A_2} := \topof{A_1}[\pi_2] + \topof{A_2}[\pi_1]$, and $\nameof{A_1 + A_2} := (\nameof{A_1},\nameof{A_2})$.

This is \stable{0}{strictly stable} in $\Gamma$, since the only part involving $\Gamma$ is the definition of $\nameof{A_1 + A_2}$, which (as an element of a \myemph{set}, not an object of a category) is, as one would hope, strictly natural in $\Gamma$.

The key point is that once the local universes $\baseof{A_i}$, $\topof{A_i}$ are chosen, the object $\baseof{A_1} \times \baseof{A_2}$ \myemph{represents} the premises of $(+)$-form: instances of the premises over a given $\Gamma$, with these universes, correspond to maps $\Gamma \myto \baseof{A_1} \times \baseof{A_2}$.
One may see $\baseof{A_1} \times \baseof{A_2}$ as parametrizing all possible sums of a type from $\baseof{A_1}$ and a type from $\baseof{A_2}$.

This pattern will be followed for all rules and constructors.
Given universes for all type premises of the rule, we construct an object $\baseof{}$ representing the rest of the data of the premises.

A specific instance of the premises over some context $\Gamma$ then corresponds to a map $\Gamma \myto \baseof{}$.
In particular, there is a universal instance over $\baseof{}$ itself.
To perform the operation on a particular instance, we first apply the operations of $\CC$ to this universal instance over $\baseof{}$.
(Here, we may rely on \stable{3}{weak stability} of structure in $\CC$ to know that reindexed from their own universes to $\baseof{}$, all types involved retain any universal properties required.)

In the case of a type constructor, we are now done, using $\baseof{}$ as the local universe of the new type, and the map $\Gamma \myto \baseof{}$ corresponding to the rest of the data as the new name map.
In the case of a term constructor, we further need to actually perform the reindexing from $\TT(\baseof{})$ to $\TT(\Gamma)$ to obtain an appropriate map in $\CC_!$.

In either case, the constructions depend on $\Gamma$ only via operations that are strictly natural in $\Gamma$: use of the universal property of $\baseof{}$, and reindexing of \myemph{maps} (not objects!) between fibers of $\TT$.
They are thus \stable{0}{strictly stable} in $\Gamma$, regardless of the stability of the structure in $\CC$.

\subsection{Manipulating local universes} \label{sec:manipulating-universes}

In the example above, it was straightforward to construct the representing object $\baseof{A} \times \baseof{B}$ for the premises of the rule in question.
For more complex rules, however, construction of this representing object---the new universe---may be rather more involved.
Indeed, this is the main technical work of the proof.

We establish here some tools for constructing such objects, beginning with one construction, in particular, which recurs for several different logical constructors.
The operations corresponding to $\Pi$-formation, $\Sigma$-formation, and $\W$-formation all take as input an object $\Gamma$, a type $A$ over $\Gamma$, and a dependent type $B$ over $\Gamma.A$.

These data in $\CC_!$, over a given object $\Gamma$, consist of $A=(\baseof{A},\topof{A},\nameof{A})$ in $\TT_{!}(\Gamma)$ and
$B=(\baseof{B},\topof{B},\nameof{B})$ in $\TT_{!}(\Gamma.A)$, which amount in $\CC$ to the following configuration:
\begin{align*}
  \begin{tikzpicture}[auto]
    \node (G) at (0,0) {$\Gamma$};
    \node (GA) at (0,1.5) {$\Gamma.A$};
    \node (NA) at (3,0) {$\baseof{A}$};
    \node (NAt) at (3,1.5) {$\topof{A}$};
    \node (NB) at (2,1.5) {$\baseof{B}$};
    \node (NBt) at (2,3) {$\topof{B}$};
    \draw[dotted] (NAt) to (NA);
    \draw[dotted] (NBt) to (NB);
    \draw[->] (G) to node[mylabelsmall] {$\nameof{A}$} (NA);
    \draw[->] (GA) to node[mylabelsmall] {$\nameof{B}$} (NB);
    \draw[->] (GA) to (G);
  \end{tikzpicture}
\end{align*}
(Here $\Gamma.A \myto A$ is the comprehension $\chi_!(A)$, taken in $\CC_!$; in terms of $\CC$, this is $\chi([A]) : \Gamma.[A] \myto \Gamma$.)

Now, fixing $(\baseof{A},\topof{A})$ and $(\baseof{B},\topof{B})$, pairs of maps $\nameof{A}$, $\nameof{B}$ as above correspond by adjunction to maps from $\Gamma$ into the object
\begin{align*}
  (\baseof{A},\topof{A}) \famop (\baseof{B},\topof{B}) & := \sum_{\baseof{A}} \prod_{\topof{A}}\left(\parbox{2.75cm}{
      \begin{tikzpicture}
        \node (GA) at (0,1.25) {$(\baseof{A}.\topof{A})\times \baseof{B}$};
        \node (G) at (0,0) {$\baseof{A}.\topof{A}$};
        \draw[->] (GA) to (G);
      \end{tikzpicture}
    }\right),
\end{align*}
where $\prod_{\topof{A}} : \CC/{\baseof{A}.\topof{A}} \myto \CC/{\baseof{A}}$ denotes the right adjoint to $\chi(\topof{A})^*$, which exists
by hypothesis \LF, and $\sum_{\baseof{A}} : \CC/\baseof{A} \myto \CC$ simply sends a map to its codomain.
(By the usual abuse of notation, we will often denote this object simply by $\baseof{A} \famop \baseof{B}$.)
Moreover, this correspondence is natural in $\Gamma$.

In particular, the identity map of $\baseof{A} \famop \baseof{B}$ corresponds to the universal such pair of maps, which we denote by 
\[ \pi_A : \baseof{A} \famop \baseof{B} \myto \baseof{A}, \qquad
   \pi_B : (\baseof{A} \famop \baseof{B}).(\topof{A}[\pi_A]) \myto \baseof{B}.\]

$\baseof{A} \famop \baseof{B}$ may thus be considered as the object of ``families of types in $\baseof{B}$, indexed by a type in $\baseof{A}$''.

The definition of $\baseof{A} \famop \baseof{B}$ may alternatively be presented in a type-theoretic internal language for $\CC$: not the arbitrary type theory that we are trying to model, but a specific type theory with just $\Pi$-types, satisfying the judgemental $\eta$-rule, to handle the substitution and exponentiation in $\CC$ and its slices.
In this language, the definition becomes:
\[ \baseof{A} \famop \baseof{B} := \exinterp{ a \of \baseof{A},\, b \of \baseof{B}^{\topof{A}(a)} }. \]

For complex constructions, this notation scales somewhat more readably than using categorical combinators.
For instance, in the operation corresponding to the $(+)$-elimination rule,
\begin{mathpar}
  \inferrule*{\Gamma \types A, B \type \\ \Gamma,\, w : {A + B} \types C(w) \type \\
              \Gamma,\, x \of A \types d_1 : C(\inl(x)) \\ 
              \Gamma,\, y \of B \types d_2 : C(\inr(y)) }
             {\Gamma,\, w : {A + B} \types \copair{x.d_1}{y.d_2}(w) : C(w) }
\end{mathpar}
once universes $\baseof{A}$, $\baseof{B}$, $\baseof{C}$ are chosen, the representing object for the premises is given by
\acmarxiv
  {\[
    \lexinterp
      a \of \baseof{A},\,
      b \of \baseof{B},\,
      c \of \baseof{C}^{\topof{A}(a) + \topof{B}(b)},\,
      d_1 \of \Pi_{x \of \topof{A}(a)}\topof{C}(c(\nu_1(x))),\,
      d_2 \of \Pi_{y \of \topof{B}(b)}\topof{C}(c(\nu_2(y)))
    \rexinterp
  \]}
  {\begin{multline*}
    \lexinterp
      a \of \baseof{A},\,
      b \of \baseof{B}, \\
      c \of \baseof{C}^{\topof{A}(a) + \topof{B}(b)},\,
      d_1 \of \Pi_{x \of \topof{A}(a)}\topof{C}(c(\nu_1(x))),\,
      d_2 \of \Pi_{y \of \topof{B}(b)}\topof{C}(c(\nu_2(y)))
    \rexinterp
  \end{multline*}}
which in categorical combinators is
\begin{multline*}
    \Sigma_{\baseof{A}} \Sigma_{\Delta_{\baseof{A}} \baseof{B}}
    \Sigma_{(\Delta_{\Delta_\baseof{A} \baseof{B}}\baseof{C})^{(\topof{A}[\pi_1] + \topof{B}[\pi_2])}} \\
    \big(
      \Pi_{\topof{A}[\pi_1 \comp \pi_1]} \topof{C}[\ev \comp \nu_1[\pi_1] ] 
    \times
      \Pi_{\topof{B}[\pi_2 \comp \pi_1]} \topof{C}[\ev \comp \nu_2[\pi_1] ] 
    \big).
\end{multline*}

For this reason, we use the internal language to present such objects below.
Unfortunately, this is somewhat laborious to formally justify.
Since we do not restrict the local universes $\baseof{A}$ to “fibrant objects” (i.e.\ with $\baseof{A} \myto 1$ a display map), nor assume that exponentiation preserves display maps, we need “types” of this internal language to include arbitrary maps of $\CC$, or at least something more general than the types of $\TT$.
$\Pi$-types between them may therefore not always be defined; so we cannot take the language to be the $\Pi$-fragment of ETT, and thus cannot quite apply Theorem~\ref{thm:hofmann-c-star} to justify its interpretation.

To thoroughly address this question, one could consider type theory with an extra judgement “$\Gamma \types A \fibration$” added to the syntax (cf.\ \acmarxiv{the system HTS of \cite{voevodsky:hts-notes}}{Voevodsky’s system HTS \cite{voevodsky:hts-notes}}), and with $\Pi$-formation restricted to the case where the domain is given by this judgement.
Correspondingly, one would consider comprehension categories equipped with a subfibration $\FF \subseteq \TT$, and extend Theorem~\ref{thm:hofmann-c-star} to this setting.

For the present purposes, however, it is simpler to regard the internal language merely as a notational shorthand, since we do not require the full interpretation function, but only finitely many instances of it, which the scrupulous reader may unwind into the algebraic language of products, pullbacks, and exponentials as required.

\subsection{Logical structure on \pdfCbang} \label{sec:main-proof}

With the machinery set up, we are now ready to lift logical structure from $\CC$ to $\CC_!$, one constructor at a time.

Taken together, the following lemmata constitute our main result:
\begin{theorem} \label{thm:main-thm}
Let $\CC$ be a full comprehension category satisfying condition \LF.

If $\CC$ has \stable{3}{weakly stable} binary sums (resp.\ $\Pi$-types, identity types, $\Sigma$-types, zero types, unit types, or $\W$-types relative to a stable class of $\Pi$-types), then its split replacement $\CC_!$ has \stable{0}{strictly stable} binary sums ($\Pi$-types, \ldots), and hence models the syntax of type theory with binary sums ($\Pi$-types, \ldots).
\end{theorem}

\counterwithin{theorem}{subsubsection}
%

See also Section~\ref{sec:further-rules} for a discussion of how this extends to other rules and constructors.

\subsubsection{Binary sums} \label{sec:sums} \label{sec:first-strux-sec}

First, we return in full to the case of binary sums.
(Note that we consider general type-theoretic (weak) binary sums, not necessarily assumed to be categorical coproducts.)

\begin{definition}
  Given a comprehension category $\CC$, an object $\Gamma \in \CC$, and types $A_1, A_2 \in \TT(\Gamma)$, a \myemph{binary sum} for $A_1$ and $A_2$ consists of:
  \begin{itemize}
  \item a type $A_1 + A_2 \in \TT(\Gamma)$;
  \item maps $\nu_i : \Gamma.A_i \myto \Gamma.A_1 + A_2$ in $\CC$ over $\Gamma$ (for $i = 1,2$) (the \myemph{sum inclusions});
  \item such that for any type $C \in \TT(\Gamma . (A_1 + A_2))$ and sections $t_i : \Gamma . A_i \myto \Gamma . A_i . C[\nu_i]$, there is some section $\copair{t_1}{t_2} : \Gamma . (A_1 + A_2) \myto \Gamma . (A_1 + A_2) . C$, such that $\copair{t_1}{t_2} \comp \nu_i = \nu_i.C \comp t_i$.
  \end{itemize}
\end{definition}

\begin{definition}
  A comprehension category $\CC$ has \myemph{\stable{3}{weakly stable} binary sums} if each $\Gamma$, $A_1$, $A_2$ as above has some binary sum $(A_1 + A_2, \nu_1, \nu_2)$, such that for every $\sigma : \Delta \myto \Gamma$, $((A_1 + A_2)[\sigma], \nu_1[\sigma], \nu_2[\sigma])$ is a binary sum for $A_1[\sigma]$ and $A_2[\sigma]$ over $\Delta$.
\end{definition}

(Note that this condition is independent of the choice of reindexings used, i.e.\ of the cleaving of $\TT$.)

\begin{definition}
  A split comprehension category $\CC$ has \myemph{\stable{0}{strictly stable} binary sums} it it is equipped with functions giving for each $\Gamma$, $A_1$, $A_2$ some chosen binary sum $(A_1 + A_2, \nu_1, \nu_2)$, and moreover for each suitable $C, t_1, t_2$ some chosen copair $\copair{t_1}{t_2}$, such that for every $\sigma : \Delta \myto \Gamma$, 
  \begin{align*}
    (A_1 + A_2)[\sigma] &= A_1[\sigma] + A_2[\sigma] \\
    \nu_i[\sigma] &= \nu_i : A_i[\sigma] \myto A_1[\sigma] + A_2[\sigma] \\
    \copair{t_1}{t_2}[\sigma] &= \copair{t_1[\sigma]}{t_2[\sigma]}.
  \end{align*}
\end{definition}

(By contrast, this is certainly not independent of the choice of reindexings given by the splitting of $\TT$.) \\

The components of this definition---the sum type, the inclusion maps, the copair map, and the copair equations---correspond precisely to the type-theoretic rules for binary sums \cite[p.~55]{martin-lof:bibliopolis}.
A split comprehension category $\CC$ with \stable{0}{strictly stable} binary sums is thus precisely what is needed to interpret the syntax of type theory with these rules (cf.~\cite[\textsection 3.3]{hofmann:syntax-and-semantics}).

\begin{lemma} \label{lemma:binary-sums}
If $\CC$ has \stable{3}{weakly stable} binary sums and satisfies condition \LF, then $\CC_!$ has \stable{0}{strictly stable} binary sums.
\end{lemma}

\begin{proof}We take each component of the definition in turn.

\paragraph{Formation} Suppose we are given suitable $\Gamma$, $A_1$, $A_2$ in $\CC_!$, and wish to form $A_1 + A_2$.
These correspond to local universes $\baseof{A_i} \in \CC$, $\topof{A_i} \in \TT(\baseof{A_i})$, and maps $\nameof{A_i} : \Gamma \myto \baseof{A_i}$.

Set $\baseof{A_1 + A_2} := \baseof{A_1} \times \baseof{A_2}$.
As indicated previously, this may be seen as the object of “(formal) sums of a type from $\baseof{A_1}$ with a type from $\baseof{A_2}$”.
Precisely, for any $\Gamma$, pairs $\nameof{A_1}$, $\nameof{A_2}$ as above correspond to maps $(\nameof{A_1},\nameof{A_2}) : \Gamma \myto \baseof{A_1 + A_2}$, naturally in $\Gamma$.

In particular, the identity map of $\baseof{A_1 + A_2}$ corresponds to the projections $\pi_i : \baseof{A_1 + A_2} \myto \baseof{A_i}$.
Pulling back the types $\topof{A_i}$ along these, we obtain types $\topof{A_i}[\pi_i] \in \TT(\baseof{A_1 + A_2})$.
Take $\topof{A_1 + A_2}$ to be the sum $\topof{A_1}[\pi_1] + \topof{A_2}[\pi_2]$.

Finally, take $\nameof{A_1 + A_2} := (\nameof{A_1},\nameof{A_2}) : \Gamma \myto \baseof{A_1 + A_2}$, picking out the appropriate specific pairs of types.

For  \stable{0}{strict stability}, suppose in addition to the above data we have some $\sigma : \Gamma' \myto \Gamma$; we need to show that $(A_1 + A_2)[\sigma] = A_1[\sigma] + A_2[\sigma]$.
It is immediate that their universes are equal---i.e.\ that $\baseof{(A_1 + A_2)[\sigma]} = \baseof{A_1[\sigma] + A_2[\sigma]}$ and $\topof{(A_1 + A_2)[\sigma]} = \topof{A_1[\sigma] + A_2[\sigma]}$---since the universe of the sum depends only on the universes of the summands, and substitution in the summands does not change their universes.
On the other hand,
\begin{align*}
  \nameof{(A_1 + A_2)[\sigma]}  & = (\nameof{A_1},\nameof{A_2}) \comp \sigma \\
  & = (\nameof{A_1} \comp \sigma, \nameof{A_2} \comp \sigma ) \\
  & = \nameof{A_1[\sigma] + A_2[\sigma]};
\end{align*}
that is, \stable{0}{strict stability} of $(+)$ comes exactly from the (strict) naturality in $\Gamma$ of the universal property of $\baseof{A_1} \times \baseof{A_2}$.

\paragraph{Introduction}  For the sum inclusions, suppose again we have $\Gamma$, $A_1$, $A_2$ in $\CC_!$.
Having constructed $A_1 + A_2$ as above, note that since $\topof{A_1 + A_2}$ was chosen as a \stable{3}{weakly stable} sum, it comes with inclusion maps $\bar{\nu}_i : \topof{A_i}[\pi_i] \myto \topof{A_1 + A_2}$.
Now set $\nu_i := \bar{\nu}_i[\nameof{A_1 + A_2}] : [A_i] \myto[] [A_1 + A_2]$.

(We are using here the convention that maps may be reindexed to arbitrary reindexings of their domain and codomain.
We will do so in future without comment.)

Breaking down this definition a little: we first consider the introduction maps in the universal case, $\bar{\nu}_i : \topof{A_i}[\pi_i] \myto \topof{A_1 + A_2}$, over $\baseof{A_1 + A_2}$.
We then reindex this to $\TT(\Gamma)$ along $\nameof{A_1 + A_2}$.

\begin{align*}
  \begin{tikzpicture}[auto,x={(5cm,0cm)},y={(-0.5cm,2.5cm)},z={(2.4cm,1.2cm)}]
    \node (G) at (0,0,0) {$\Gamma$};
    \node (Vsum) at (1,0,0) {$\baseof{A_1 + A_2}$};
    \node (VA) at (1.8,0,0) {$\baseof{A_i}$};
    \draw[->] (G) to node[mylabel] {$\nameof{A_1 + A_2}$} (Vsum);
    \draw[->] (Vsum) to node[mylabel] {$\pi_i$} (VA);
    \node (GA) at (0,1,-0.5) {$\Gamma.A_i$};
    \draw[->] (GA) to (G);
    \node (EApi) at (1,1,-0.5) {$\baseof{A_1 + A_2} . \topof{A_i}[\pi_i]$};
    \draw[->] (EApi) to (Vsum);
    \node (EA) at (1.8,1,-0.5) {$\baseof{A_i}.\topof{A_i}$};
    \draw[->] (EA) to (VA);
    \node (Gsum) at (0,1,0.5) {$\Gamma.(A_1 + A_2)$};
    \draw[->] (Gsum) to (G);
    \node (Esum) at (1,1,0.5) {$\baseof{A_1 + A_2} . \topof{A_1 + A_2}$};
    \draw[->] (Esum) to (Vsum);
    \draw[->,cross line] (Gsum) to (Esum);
    \draw[->,cross line] (GA) to (EApi);
    \draw[->,cross line] (EApi) to (EA);
    \draw[->,dashed] (GA) to node[mymidlabel] {$\nu_i$} (Gsum);
    \draw[->] (EApi) to node[mymidlabel] {$\bar{\nu}_i$} (Esum);
  \end{tikzpicture}
\end{align*}

\paragraph{Elimination}  Defining copairing in $\CC_!$ holds a subtle pitfall for the unwary---one worth looking at explicitly, since it will recur later for other constructors.

Consider the data $\Gamma$, $A_1$, $A_2$, $C$, $d_1$, $d_2$ in $\CC_!$, for forming a copair.
In $\CC$, these correspond to $\Gamma$, $\baseof{A_i}$, $\topof{A_i}$, $\nameof{A_i}$ as above, together with another local universe $\topof{C} \in \TT(\baseof{C})$, a name map $\nameof{C} : \Gamma . (A_1 + A_2) \myto \baseof{C}$, and sections $d_i : \Gamma.A_i  \myto \Gamma . A_i . \topof{C}[\nameof{C} \comp \nu_i]$.

We require a copair for $C$, $d_1$, $d_2$ in $\CC_!$; that is, a certain section $\Gamma.(A_1 + A_2) \myto \Gamma.(A_1 + A_2).C$, commuting appropriately with $\nu_i$, $d_i$.
This corresponds to a section $\Gamma.[A_1 + A_2] \myto \Gamma.[A_1 + A_2].[C]$ in $\CC$, commuting there with $\nu_i$, $d_i$.

The obvious way to obtain such a section is simply as a copair in $\CC$.
We chose $\topof{A_1+A_2}$ as a weakly stable $(+)$-type over $\baseof{A_1+A_2}$, so $[A_1 + A_2]$ is a $(+)$-type over $\Gamma$, and the data $[C]$, $d_1$, $d_2$ are just right for forming a copair there.

However, the resulting operation would \myemph{not} necessarily be \stable{0}{strictly stable}, since the copair in $\CC$ was taken over $\Gamma$.%
\footnote{If the $(+)$-types of $\CC$ are \stable{1}{pseudo-stable}---for instance, if they satisfy the $\eta$-rule, making them categorical coproducts---then this direct definition of copairing is \stable{0}{strictly stable} after all; and in the case of $(+)$-types, this would be a reasonably mild extra condition to demand.
However, for identity types (and more general inductive families), the analogous hypothesis would be much less innocuous, implying in particular the reflection rule, and hence UIP.}
We therefore resist this tempting shortcut and keep to the general approach prescribed above, first taking a ``universal copair'' depending just on the universes $\baseof{A_1}$, $\baseof{A_2}$, $\baseof{C}$.
Only having done this do we pull it back (strictly naturally) to the specific context $\Gamma$ in question.

Precisely, fix universes $\baseof{A_i}$, $\topof{A_i}$, $\baseof{C}$, $\topof{C}$, and set:
\begin{equation*} \begin{split}
    \baseof{\copair{d_1}{d_2}} :=
    \lexinterp
  &   a_1\of \baseof{A_1},\, a_2 \of \baseof{A_2},\,
      c \of \baseof{C}^{\topof{A_1+A_2}(a_1,a_2)},\, \\
  &   d_1 \of \Pi_{x \of \topof{A_1}(a_1)}\topof{C}(c(\nu_1(x))),\,
      d_2 \of \Pi_{y \of \topof{A_2}(a_2)}\topof{C}(c(\nu_2(y)))
    \rexinterp
\end{split} \end{equation*}

The remaining data $\nameof{A_i}$, $\nameof{C}$, $d_i$ correspond to maps $\Gamma \myto \baseof{\copair{d_1}{d_2}}$, naturally in $\Gamma$.
In particular, the identity $1_{\baseof{\copair{d_1}{d_2}}}$ corresponds to maps
\begin{gather*}
  \pi_{A_i} : \baseof{\copair{d_1}{d_2}} \myto \baseof{A_i} \quad \pi_{C} : \baseof{\copair{d_1}{d_2}} . \topof{A_1 + A_2}[(\pi_{A_1},\pi_{A_2})] \myto \baseof{C} \\
  \pi_{d_i} : \baseof{\copair{d_1}{d_2}}.\topof{A_i}[\pi_{A_i}] \myto \baseof{\copair{d_1}{d_2}}.\topof{A_i}[\pi_{A_i}].\topof{C}[\pi_{C} \comp \nu_i[(\pi_{A_1},\pi_{A_2})]].
\end{gather*}

Now, as in the direct approach, since $\topof{A_1 + A_2}$ was a \stable{3}{weakly stable} sum, its reindexing $\topof{A_1 + A_2}[(\pi_{A_1},\pi_{A_2})]$ together with the inclusion maps $\nu_i[(\pi_{A_1},\pi_{A_2})]$ is a sum for $\topof{A_1}[\pi_{A_1}]$ and $\topof{A_2}[\pi_{A_2}]$ over $\baseof{\copair{d_1}{d_2}}$.
So we may form there the copair section
\[ \! \copair{\pi_{d_1}}{\pi_{d_2}}
 : \baseof{\copair{d_1}{d_2}} . \topof{A_1 + A_2}[(\pi_{A_1},\pi_{A_2})]
  \myto \baseof{\copair{d_1}{d_2}} . \topof{A_1 + A_2}[(\pi_{A_1},\pi_{A_2})] . \topof{C}[\pi_C]. \]

Pulling this back along 
\[ (\nameof{A_1},\nameof{A_2},\nameof{C},d_1,d_2) : \Gamma \myto \baseof{\copair{d_1}{d_2}} \]
then gives us a section
\[ \Gamma . (A_1 + A_2) \myto \Gamma . (A_2 + A_2) . C \]
which we take as the copair $\copair{d_1}{d_2}$ in $\CC_!$.

\stable{0}{Strict stability} of this operation follows from the fact that the only involvement of $\Gamma,\nameof{A_1},\nameof{A_2},\nameof{C},d_1,d_2$ was via the map $\Gamma \myto \baseof{\copair{d_1}{d_2}}$, and the pullback of a section along this map, both of which are strictly natural in $\Gamma$.
In particular, the copair used in $\CC$, which a priori may not satisfy any naturality condition, was taken over $\baseof{\copair{d_1}{d_2}}$, and so is unaffected by any reindexing $\Gamma' \myto \Gamma$.

\paragraph{Computation}
Finally, the copair-inclusion equations for $\copair{d_1}{d_2}$ follow directly \acmarxiv{from the equations in $\CC$ for the copair}{from the equations in $\CC$ for} $\copair{\pi_{d_1}}{\pi_{d_2}}$, pulled back along $(\nameof{A_1},\nameof{A_2},\nameof{C},d_1,d_2)$.
\end{proof}

\subsubsection{Dependent products}\label{sec:pis}

\begin{definition} \label{def:dep-products}
  Given $\Gamma \in \CC$, $A \in \TT(\Gamma)$, $B \in \TT(\Gamma.A)$, a \myemph{dependent product} for $\Gamma$, $A$, $B$ is given by:
  \begin{itemize}
  \item a type $\prod[A,B] \in \TT(\Gamma)$;
  \item a map $\app_{A,B} : \prod[A,B][\chi(A)] \myto B$ in $\TT(\Gamma.A)$;
  \item an operation giving for each section $t : \Gamma.A \myto \Gamma.A.B$, a section $\lambda (t) : \Gamma \myto \Gamma.\prod[A,B]$, such that $(\Gamma.A.\app_{A,B}) \comp (1_{\Gamma.A}, \lambda(t)) = t$.
  \end{itemize}
\end{definition}

\begin{definition}
  $\CC$ has \myemph{\stable{3}{weakly stable} dependent products} if every $\Gamma$, $A$, $B$ has some $(\prod[A,B],\app_{A,B})$ as above, such that for every $\sigma : \Delta \myto \Gamma$, there is some operation $\lambda$ making $(\prod[A,B][\sigma],\app_{A,B}[\sigma],\lambda)$ a dependent product for $\Delta$, $A[\sigma]$, $B[\sigma]$.
  More specifically, call such $(\prod[A,B],\app_{A,B})$ a \myemph{\stable{3}{weakly stable} dependent product} for $\Gamma$, $A$, $B$.
\end{definition}

(Note again that this is independent of the cleaving of $\TT$.)
 
\begin{definition}
$\CC$ has \myemph{\stable{0}{strictly stable} dependent products} if it is equipped with operations giving $\prod(A,B)$, $\app_{A,B}$, and $\lambda (t)$ for all appropriate $\Gamma,A,B$ and $t$ as above, such that for every $\sigma : \Delta
\myto \Gamma$,
\begin{align*}
  \prod[A,B][\sigma] & = \prod\bigl[A[\sigma],B[\sigma]\bigr]\\
  \app_{A,B}[\sigma] & = \app_{A[\sigma],B[\sigma]}\\
  ( \lambda (t) )[\sigma] & = \lambda (t[\sigma]).
\end{align*}
\end{definition}

(Again, this is by contrast entirely dependent on the chosen cleaving.)
 
\begin{lemma} \label{lemma:pi-types}
  If $\CC$ has \stable{3}{weakly stable} dependent products and satisfies condition \LF, then $\CC_{!}$ has \stable{0}{strictly stable} dependent products.
\end{lemma}

\begin{proof} Again, we consider the components of the definition---the four rules for $\Pi$-types---one by one.
 
\paragraph{Formation}
Suppose we \acmarxiv{are given}{have} $A=(\baseof{A},\topof{A},\nameof{A})$ in $\TT_{!}(\Gamma)$ and $B=(\baseof{B},\topof{B},\nameof{B})$ in $\TT_{!}(\Gamma.A)$, and wish to form $\prod[A,B]$.

We begin by setting
\begin{align*}
  \baseof{\prod[A,B]} & := \baseof{A} \famop \baseof{B}.
\end{align*}
(Recall \acmarxiv{that }{}$\baseof{A} \famop \baseof{B}$ is the object \acmarxiv{described}{given} in the internal language by $\exinterp{ a \of \baseof{A},\, b \of \baseof{B}^{\topof{A}(a)} }$.)

As described in Section~\ref{sec:manipulating-universes}, maps $\Gamma \myto \baseof{A} \famop \baseof{B}$ correspond to data $\nameof{A} : \Gamma \myto \baseof{A}$, $\nameof{B} : \Gamma.\topof{A}[\nameof{A}] \myto \baseof{B}$ as above.
In particular, there is the universal case
\[ \pi_A : \baseof{A} \famop \baseof{B} \myto \baseof{A}, \qquad
   \pi_B : (\baseof{A} \famop \baseof{B}).(\topof{A}[\pi_A]) \myto \baseof{B} \]
over $\baseof{\prod[A,B]}$ itself.
To obtain $\topof{\prod[A,B]}$, we choose a \stable{3}{weakly stable} dependent product in $\CC$ for this universal case:
\[  \topof{\prod[A,B]} := \prod \left[ \topof{A}[\pi_A],\topof{B}[\pi_B] \right]
    \qquad \in \TT(\baseof{\prod[A,B]}).
\]

Finally, we take $\nameof{\prod[A,B]} : \Gamma \myto \baseof{\prod[A,B]}$ to be the map corresponding to the pair $\nameof{A}$, $\nameof{B}$ under the universal property of $\baseof{\prod[A,B]}$. \\

Together, these define the type $\prod[A,B]$ in $\CC_!$.
To see that the resulting operation is moreover \stable{0}{strictly stable}, suppose we have $\Gamma$, $A$, $B$ as above, and additionally some $\sigma : \Gamma' \myto \Gamma$.
We need to check that $\prod[A,B][\sigma] = \prod[A[\sigma],B[\sigma]]$.

It is immediate that the local universes of these two products are the same, since they depend only on the local universes $\baseof{A}$, $\baseof{B}$, which are unaffected by the reindexing.

It therefore only remains to show that $\nameof{\prod[A,B][\sigma]} = \nameof{\prod[A[\sigma],B[\sigma]]}$; but this follows just from the (strict) naturality in $\Gamma$ of the universal property of $\baseof{A} \famop \baseof{B}$.

\paragraph{Application}

By the definition of $\topof{\prod[A,B]}$ as a \stable{3}{weakly stable} dependent product, it comes with a
map
\[ \app_{\topof{A}[\pi_A],\topof{B}[\pi_B]}
 : \baseof{\prod[A,B]}.\topof{A}[\pi_A].\topof{\prod[A,B]}
  \myto \baseof{\prod[A,B]}.\topof{A}[\pi_A].\topof{B}[\pi_B] \]
over $\baseof{\prod[A,B]}.\topof{A}[\pi_A]$.

We define $\app_{A,B}$ just as the reindexing of this map to $\TT(\Gamma.A)$:
\begin{align*}
  \begin{tikzpicture}[auto]
    \node (GAB) at (0,3) {$\Gamma.A.B$};
    \node (GA) at (0,0) {$\Gamma.A$};
    \node (VAB) at (7,3) {$\baseof{\prod[A,B]}.\topof{A}[\pi_A].\topof{B}[\pi_B] $};
    \node (VA) at (7,0) {$\baseof{\prod[A,B]}.\topof{A}[\pi_A].$};
    \node (VAP) at (4.5,1.25) {$\baseof{\prod[A,B]}.\topof{A}[\pi_A].\topof{\prod[A,B]}$};
    \node (GAP) at (-2.5,1.25) {$\Gamma.A.\prod[A,B]$};
    \draw[->] (GAB) to (GA);
    \draw[->] (GA) to node[mylabel] {$\nameof{\prod[A,B]}.\topof{A}$} (VA);
    \draw[->] (VAB) to (VA);
    \draw[->] (GAB) to (VAB);
    \draw[->,cross line] (GAP) to (VAP);
    \draw[->,bend right=15] (GAP) to (GA);
    \draw[->,bend right=15] (VAP) to (VA);
    \draw[->] (VAP) to node[mymidlabel] {$\app_{\topof{A}[\pi_A],\topof{B}[\pi_B]}$} (VAB);
    \draw[->,dashed] (GAP) to node[mymidlabel] {$\app_{A,B}$} (GAB);
  \end{tikzpicture}
\end{align*}

Once again, \stable{0}{strict stability} of this follows directly by construction.
The universal case $\app_{\topof{A}[\pi_A],\topof{B}[\pi_B]}$ depends only on $\baseof{A}$, $\baseof{B}$; and the subsequent reindexing is strictly natural in $\Gamma$.

\paragraph{Introduction}
As in the case of copairing above, there is a direct approach to defining $\lambda$-abstraction, which however may fail to be \stable{0}{strictly stable}.
We therefore take once again a two-stage approach.
First, we define the object $\baseof{\lambda_{A,B}}$ of all possible $\lambda$-abstractions into $\topof{\prod{A,B}}$, and choose a universal $\lambda$-abstraction over that; then, we pick out the $\lambda$-abstractions in $\CC_{!}$ as pullbacks of that universal one.

Let $\baseof{\lambda_{A,B}}$ be the object:
\[ \exinterp{
  a : \baseof{A},\,
  b : \baseof{B}^{\topof{A}(a)},\, 
  t : \exprod_{x \of \topof{A}(a)}\topof{B}(b(x))
}\]
Here we write $\exprod$ to emphasize that this description is interpreted using the categorical exponentials in $\CC$ provided by condition \LF, not the \stable{3}{weakly stable} dependent products of types in $\TT$ used for $\topof{\prod[A,B]}$.
Modulo that difference, this is exactly analogous to the object $\baseof{\prod[A,B]}.\topof{\prod[A,B]}$.

Write $\pi_{\lambda_{A,B}}$ for the evident projection $\baseof{\lambda_{A,B}} \myto \baseof{\prod[A,B]}$.

As a categorical dependent product, $\baseof{\lambda_{A,B}}$ has application map 
\[ \exapp : \baseof{\lambda_{A,B}} .  \topof{A}[\pi_A \comp \pi_{\lambda_{A,B}}] 
  \myto
 \baseof{\lambda_{A,B}} .  \topof{A}[\pi_A \comp \pi_{\lambda_{A,B}}] . \topof{B}[\pi_A \comp \pi_{\lambda_{A,B}}.\topof{A}]. \]

Since $\topof{\prod[A,B]}$ was a weakly stable dependent product, $\exapp$ induces a map 
\[ \lambda(\exapp) : \baseof{\lambda_{A,B}}
  \myto \baseof{\lambda_{A,B}} . \topof{\prod[A,B]}[\pi_{\lambda_{A,B}}] \]
with $\app_{\topof{A}[\pi_A],\topof{B}[\pi_B]}[\pi_{\lambda_{A,B}}] \comp \lambda(\exapp)[\chi(\topof{A}[\pi_A \comp \pi_{\lambda_{A,B}}])] = \exapp$.

Now, suppose we are given the inputs for $\lambda$-abstraction in $\CC_!$.
That is, in addition to $\Gamma$, $A$, $B$ as before, we have a section $t : \Gamma.A \myto \Gamma.A.B$.

By the universal property of the categorical dependent product, such tuples $(\nameof{A},\nameof{B},t)$ correspond naturally to maps $\nameof{(A,B,t)} : \Gamma \myto \baseof{\lambda_{A,B}}$.
So, we may take $\lambda(t) : \Gamma \myto \Gamma.\prod[A,B]$ to be the reindexing of $\lambda(\exapp)$ along $\nameof{(A,B,t)}$:
\begin{align*}
  \begin{tikzpicture}[auto]
    \node (G) at (0,0) {$\Gamma$};
    \node (G2) at (-2.5,1.25) {$\Gamma$};
    \node (GP) at (0,3) {$\Gamma.\prod[A,B]$};
    \node (V) at (7,0) {$\baseof{\lambda_{A,B}}$};
    \node (V2) at (4.5,1.25) {$\baseof{\lambda_{A,B}}$};
    \node (VP) at (7,3) {$\baseof{\lambda_{A,B}} . \topof{\prod[A,B]}[\pi_{\lambda_{A,B}}]$};
    \draw[->] (GP) to (G);
    \draw[->] (G) to node[mylabel] {$\nameof{(A,B,t)}$} (V);
    \draw[->] (VP) to (V);
    \draw[->] (GP) to (VP);
    \draw[->,cross line] (G2) to (V2);
    \draw[->,bend right=15] (G2) to node[mylabel,swap] {$1_{\Gamma}$} (G);
    \draw[->,bend right=15] (V2) to (V);
    \draw[->] (V2) to node[mymidlabel] {$\lambda(\exapp)$} (VP);
    \draw[->,dashed] (G2) to node[mymidlabel] {$\lambda(t)$} (GP);
  \end{tikzpicture}
\end{align*}

\stable{0}{Strict stability} is immediate by construction, just as for $\app_{A,B}$.

\paragraph{Computation} Finally, the $\beta$-reduction equation for $\lambda(t)$ in $\CC_!$ follows from the corresponding equation for the universal case $\lambda(\exapp)$, which holds by its construction as a $\lambda$-abstraction into $\topof{\prod[A,B]}$.
\end{proof}

Often, one may want to restrict the $\Pi$-types used to some well-behaved or well-understood subclass---typically, the categorical dependent products.
Indeed, one might want the same for other constructors as well; we spell out the case of $\Pi$-types since we will need it for setting up weakly stable $\W$-types below.

\begin{definition}
  A \myemph{stable class of $\Pi$-types} on $\CC$ consists of:
  \begin{itemize}
  \item for each $\Gamma$, $A$, $B$, a non-empty family $\GG_\Pi(\Gamma,A,B)$ of $\Pi$-types $(\Pi,\app)$ for $A$, $B$, stable under reindexing, in that for all $\sigma : \Delta \myto \Gamma$ and $(\Pi,\app) \in \GG_\Pi(\Gamma,A,B)$, and any reindexings $A'$, $B'$, $\Pi'$, $\app'$ of these along $\sigma$, we have $(\Pi',\app') \in \GG_\Pi(\Gamma',A',B')$;
  \item and, for each $\Gamma$, $A$, $B$ as before, $(\Pi,\app) \in \GG_\Pi(\Gamma,A,B)$, and section $t : \Gamma.A \myto \Gamma.A.B$, a non-empty family of sections $\GG_{\lambda}(\Gamma,\ldots,t)$, similarly stable under reindexing.
  \end{itemize}

  Given such a class, we refer to an element of $\GG_\Pi(\Gamma,A,B)$ (resp.\ $\GG_{\lambda}(\Gamma,\ldots,t)$) as a \myemph{good} $\Pi$-type for $A$, $B$ (resp.\ a good $\lambda$-abstraction of $t$).
  Stability says just that any reindexing of a good $\Pi$-type or $\lambda$-abstraction is again good.
\end{definition}

\begin{scholium} \label{sch:good-pi-types}
  If $\CC$ is equipped with a stable class of dependent products, then $\CC_!$ has \stable{0}{strictly stable} dependent products, always chosen  from the given stable class. 
\end{scholium}

\begin{proof}
  Immediate from the proof of Lemma~\ref{lemma:pi-types}.
\end{proof}

\begin{proposition} \label{prop:stable-class-from-pseudo-stable}
$\CC$ has \stable{3}{weakly stable} $\Pi$-types if and only if it can be equipped with a stable class of $\Pi$-types.
\end{proposition}

\begin{proof}
  The ``if'' direction is immediate.
  For the ``only if'', note that if $\CC$ has \stable{3}{weakly stable} $\Pi$-types, then the class of \myemph{all} \stable{3}{weakly stable} $\Pi$-types and $\lambda$-abstractions forms a stable class.
\end{proof}

Finally, we pause to consider the \stable{1}{pseudo-stable} level, again for later use in the presentation of $\W$-types:

\begin{definition}
  $\CC$ has \myemph{\stable{1}{pseudo-stable} dependent products} if it is equipped with operations $\Pi$, $\app$, $\lambda$ as above, together with a cartesian functorial action on cartesian maps; that is, for each $\sigma : \Gamma' \myto \Gamma$, and cartesian maps $\sigma_A : A' \to A$ over $\sigma$ and $\sigma_B : B' \to B$ over $\sigma.\sigma_A$, a cartesian map $\Pi[\sigma_A,\sigma_B] : \Pi[A',B'] \to \Pi[A,B]$ over $\sigma$, such that
  \begin{align*}
    \Pi[1_A,1_B] & = 1_{\Pi[A,B]} \\
    \Pi[\tau_A \comp \sigma_A,\tau_B \comp \sigma_B] & = \Pi[\tau_A,\tau_B] \comp \Pi[\sigma_A,\sigma_B] \\
    \sigma.\sigma_B \comp \app_{A',B'} & = \app_{A,B} \comp \sigma.\sigma_A.\Pi[\sigma_A,\sigma_B] \\
    \sigma.\Pi[\sigma_A,\sigma_B] \comp \lambda_{A',B'}(t[\sigma]) & = \lambda_{A,B}(t) \comp \sigma
  \end{align*}
  (for all suitable $\sigma, \sigma_A, \sigma_B, \tau, \tau_A, \tau_B, t$).
\end{definition}

\begin{proposition} \label{prop:good-class-from-pseudo-stable}
  If $\CC$ is equipped with pseudo-stable dependent products, then it carries a stable class of dependent products, consisting of all $\Pi$-types equipped with isomorphisms to those supplied by the pseudo-stable structure, and all $\lambda$-abstractions corresponding under those isomorphisms to the ones provided by the pseudo-stable structure. \qed
\end{proposition}

\subsubsection{Identity types}

The identity types we consider in this section will be slightly stronger than those set out in Section~\ref{sec:stability-conditions} above.
Specifically, we consider structure corresponding to the elimination rule:
\begin{mathpar}
  \inferrule*{\Gamma,\ x,y \of A,\ u \of \id[A](x,y),\ \Delta \types C(x,y,u) \type \\
              \Gamma,\ z \of A,\ \Delta[z/x,z/y,\rr(z)/u] \types d(z):C(z,z,\rr(z)) }
             {\Gamma,\ a,b \of A,\ p \of \id[A](x,y),\ \Delta[a/x,b/y,p/u] \types \Jelim[x,y,u.C;\, z. d](a,b,p) : C(a,b,p)}
\end{mathpar}
where $\Delta$ may be an arbitrary context extension.

Often, $\Delta$ is omitted in the basic definition of identity types, and the version with it is called \myemph{strong} or \myemph{Frobenius} identity types.
In the presence of $\Pi$-types, the two forms are interderivable, so the weak form suffices.
In general, however, the Frobenius form is the more important and more natural; so that is the form we consider here, and throughout this section, \myemph{identity types} will always refer to the Frobenius form.
(As ever, though, the construction works for either set of rules.)

Given this, we make some slight modifications to the definitions of Sec.~\ref{sec:stability-conditions}.

\begin{definition}
  A \myemph{(Frobenius) identity type} for $\Gamma \in \CC$, $A \in \TT(\Gamma)$ consists of $\id[A]$, $\rr_A$ as in Definition~\ref{def:weak-id-types}, together with for each sequence $\Delta = (B_1,\ldots,B_n)$ such that $B_i \in \TT(\Gamma.A.A.\id[A].B_1. \,\ldots\, .B_{i-1})$, each $C\in\TT(\Gamma.A.A.\id[A].\Delta)$, and each $d:\Gamma.A.\Delta[\rr_A] \myto \Gamma.A.A.\id[A].\Delta.C$ such that the square
  \begin{align*}
    \begin{tikzpicture}[auto]
      \node (UL) at (0,1.75) {$\Gamma.A.\Delta[\rr_A]$};
      \node (BL) at (0,0) {$\Gamma.A.A.\id[A].\Delta$};
      \node (BR) at (3.3,0) {$\Gamma.A.A.\id[A].\Delta$};
      \node (UR) at (3.3,1.75) {$\Gamma.A.A.\id[A].\Delta.C$};
      \draw[->] (UL) to node[mylabel] {$d$} (UR);
      \draw[->] (UL) to node[mylabel,swap] {$\rr_{A}$} (BL);
      \draw[->] (BL) to node[mylabel,swap] {$1$} (BR);
      \draw[->] (UR) to (BR);
      \draw[->,dashed] (BL) to node[mymidlabel] {$\jj_{A,\Delta,C,d}$} (UR);
    \end{tikzpicture}
  \end{align*}
  commutes, a diagonal filler $\jj_{A,\Delta,C,d}$, which we call an \myemph{elimination section} for the data $A,\Delta,C,d$.

  A choice of identity types on $\CC$ is \myemph{\stable{0}{strictly stable}} if for each $\sigma : \Gamma' \myto \Gamma$, and all appropriate $A$, $\Delta$, $C$, $d$,
  \begin{align*}
    \id[A][\sigma] & = \id[{A[\sigma]}]\\
    \rr_{A}[\sigma] & = \rr_{A[\sigma]}\\
    \jj_{A,\Delta,C,d}[\sigma] & = \jj_{A[\sigma],\Delta[\sigma],C[\sigma],d[\sigma]}.
  \end{align*}

  A \myemph{\stable{3}{weakly stable} identity type} for $A \in \TT(\Gamma)$ is $(\id,\rr)$ as above such that, for all $\sigma:\Gamma' \myto \Gamma$, there exists some $\jj$ making $(\id[][\sigma],\rr[\sigma],\jj)$ an identity type for $A[\sigma]$.
 Say $\CC$ \myemph{has \stable{3}{weakly stable} identity types} if for each $\Gamma$, $A$, there exists some \stable{3}{weakly stable} identity type.
\end{definition}

\begin{lemma} \label{lemma:identity-types}
If $\CC$ has \stable{3}{weakly stable} identity types and satisfies condition \LF, then $\CC_!$ has \stable{0}{strictly stable} identity types.
\end{lemma}

\begin{proof} As usual, we consider the rules one by one.

\paragraph{Formation}

Given $A \in\TT_{!}(\Gamma)$, choose some \stable{3}{weakly stable} identity type $(\id[\topof{A}],\rr_{\topof{A}})$ for $\topof{A}$ over $\baseof{A}$, and define:
\begin{align*}
  \baseof{\id[A]} & := \baseof{A}.\topof{A}.\topof{A} \\
  \topof{\id[A]} & := \id[\topof{A}] \\
  \nameof{\id[A]} & := \nameof{A}.\topof{A}.\topof{A} : \Gamma.A.A \myto \baseof{\id[A]}
\end{align*}

As usual, this is \stable{0}{strictly stable} just since $\baseof{\id[A]}$ and $\topof{\id[A]}$ do not depend on $\Gamma$, $\nameof{A}$, while the construction of $\nameof{\id[A]}$ is strictly natural in $\Gamma$.

\paragraph{Introduction}

For the reflexivity map, take 
\[ \rr_A := \rr_{\topof{A}}[\nameof{A}] \]
where $\rr_{\topof{A}}$ is the reflexivity map of the chosen \stable{3}{weakly stable} identity type $\id[\topof{A}]$.

Again, \stable{0}{strict stability} is immediate.

\paragraph{Elimination, computation}

Let $\Gamma$, $A$, $\Delta = (B_1,\ldots,B_n)$, $C$, $d$ be instances of the premises of $\id$-elimination in $\CC_!$.
(\acmarxiv{So, in}{In} particular, $B_i \in \TT_!(\Gamma.A.A.\id[A].B_1.\ldots.B_{i-1})$.)

As usual, we work by first fixing the universes $\baseof{A}$, $\baseof{B_1}, \ldots, \baseof{B_n}$, $\baseof{C}$, and constructing a representing object $\baseof{}$ for ``data as above, with the given universes''.
Due to the Frobenius premise $\Delta$, this is slightly more involved than other rules we have considered.

In the internal language, it may be expressed as
\begin{equation*}
  \begin{split}
    \baseof{} := \lexinterp
      & a : \baseof{A}, \\
      & b_1 : \prod {x,x' \of \topof{A}(a),\, y \of \id[\topof{A}](a,x,x')}.\ \baseof{B_1}, \\
      & b_2 : \prod {x,x' \of \topof{A}(a),\, y \of \id[\topof{A}](a,x,x'),\,
                       z_1 \of \topof{B_1}(b_1(x,x',y))}.\ \baseof{B_2},  \\
      & \ldots \\
      & c : \prod {x,x',y,z_1,\ldots,z_n}.\ \baseof{C} \\  
      & d : \prod {x : \topof{A},\, z_1 : \topof{B_1}(b_1(x,x,\rr_{\topof{A}}(a,x))),\ \ldots,\ }\\
      & \phantom{d : \prod {x : \topof{A}},}\ {z_n : \topof{B_n}(b_n(x,x,\rr_{\topof{A}}(a,x),z_1,\ldots,z_{n-1}))},\ \\
      & \phantom{d : }\ \topof{C}(c(x,x,\rr_{\topof{A}}(a,x),z_1,\ldots,z_n)) \rexinterp 
  \end{split}
\end{equation*}
(Here $(\id[\topof{A}],\rr_{\topof{A}})$ are the identity type used for $\id[A]$, $\rr_A$ above.)

Maps $\nameof{A,\Delta,C,d} : \Gamma \myto \baseof{}$ now correspond, naturally in $\Gamma$, to \acmarxiv{tuples $(\nameof{A},\nameof{B_1},\ldots,\nameof{B_n},\nameof{C},d)$ over $\Gamma$}{tuples over $\Gamma$ $(\nameof{A},\nameof{B_1},\ldots,\nameof{B_n},\nameof{C},d)$} as in the original data.
In particular, the identity $1_{\baseof{}}$ corresponds to such data over $\baseof{}$ itself.
Since $(\id[\topof{A}],\rr_{\topof{A}})$ was \stable{3}{weakly stable}, we may choose some universal elimination section $\jj$ for this data:
\begin{multline*}
\jj : \exinterp{ (a,b_1,\ldots,c,d) \of \baseof{},\ x,x'\of \topof{A}(a),\, y \of \id[\topof{A}](a,x,x'),\,z_1,\ldots,z_n } \\
  \myto \exinterp{ (a,b_1,\ldots,c,d),\, x,x',y,\, z_1, \ldots, z_n,\, c \of \topof{C}(c(x,\ldots,z_n))}.
\end{multline*}

Returning to the original specific inputs $\Gamma$, $A$, $\Delta$, $C$, $d$, we can now pull this universal $\jj$ back along the representing map $\nameof{A,\Delta,C,d} : \Gamma \myto \baseof{}$ to give the required elimination section for $d$:
\[ \jj_{A,\Delta,C,d} := \jj[\nameof{A,\Delta,C,d}] : (\Gamma.A.A.\id[A].\Delta) \myto (\Gamma.A.A.\id[A].\Delta.C). \]

\stable{0}{Strict stability} follows, as usual, from the fact that this depends on $\Gamma$ only via the universal property of $\baseof{}$ and the action on morphisms of a pullback functor, both of which are suitably natural.

In particular, the choice of an elimination section $\jj$---the one operation which is \myemph{not} strictly natural in many models, and cannot easily be made so---was made over $\baseof{}$, and so depends only on the universes involved, not on $\Gamma$, or on anything else affected by reindexing in $\CC_!$.
\end{proof}

\subsubsection{Other constructors} \label{sec:other-constructors} \label{sec:last-strux-sec}

The three cases above illustrate essentially all the issues that arise in constructing structure on $\CC_!$.

For the remaining constructors, therefore, we give just the definitions of the appropriate \stable{0}{strictly}/\stable{3}{weakly stable} structure, and precise statements of the lifting lemmas.
We omit their proofs, since they follow exactly the same template as the cases above.

The definitions, too, contain just the same components as the cases above, with one exception, in the case of $\W$-types.
Since their rules refer to $\Pi$-types, the definition of \stable{3}{weakly stable} $\W$-types must be given relative to some form of $\Pi$-types---most naturally and flexibly, to a chosen stable class thereof.
This is the only new twist appearing in the definitions below, and indicates more generally how one might extend the present results to other type-formers whose rules make reference to other previously-defined types.

\paragraph{Dependent sums}

\begin{definition}
  For $\Gamma \in \CC$,  $A \in \TT(\Gamma)$, and $B \in \TT(\Gamma.A)$, a \myemph{dependent sum} for $B$ consists of:
  \begin{itemize}
  \item a type $\Sigma_A B \in \TT(\Gamma)$;
  \item a “pairing” map $\dpair{-}{-} : \Gamma.A.B \myto \Gamma.\Sigma_A B$; \opt{arxiv}{such that}
  \item \acmarxiv{such that }{}for any type $C \in \TT(\Gamma . \Sigma_A B)$ and section $d : \Gamma.A.B \myto \Gamma.A.B.C[\dpair{-}{-}]$, there is a section $\dsplit{C}{d} : \Gamma . \Sigma_A B \myto \Gamma . \Sigma_A B . C$, with $\dsplit{C}{d} \comp \dpair{-}{-} = (\dpair{-}{-}.C) \comp d$.
  \end{itemize}

  $\CC$ has \myemph{\stable{3}{weakly stable} dependent sums} if \acmarxiv{each $\Gamma$, $A$, $B$ as above has some}{for each $\Gamma$, $A$, $B$ as above, there exists some} $(\Sigma_A B, \dpair{-}{-})$, such that for each $\sigma : \Delta \myto \Gamma$, $(\Sigma_A B[\sigma], \dpair{-}{-}[\sigma])$ is a dependent sum for $A[\sigma]$ and $B[\sigma]$ over $\Delta$.

  A split comprehension category $\CC$ has \myemph{\stable{0}{strictly stable} dependent sums} if it is equipped with functions giving for each $\Gamma$, $A$, $B$ a \opt{acm}{chosen} dependent sum $(\Sigma_A B, \dpair{-}{-})$, and moreover for each suitable $C, d$ some appropriate section $\dsplit{C}{d}$, all commuting on the nose with reindexing in $\CC$.
\end{definition}

\begin{lemma} \label{lemma:dependent-sums}
If $\CC$ has \stable{3}{weakly stable} dependent sums and satisfies condition \LF, then $\CC_!$ has \stable{0}{strictly stable} dependent sums. \qed
\end{lemma}

\paragraph{Zero types}

\begin{definition}
  Given  $\Gamma \in \CC$, a \myemph{zero type over $\Gamma$} consists of:
  \begin{itemize}
  \item a type $\zero \in \TT(\Gamma)$;
  \item for any type $C \in \TT(\Gamma.\zero)$, a section of $C$.
  \end{itemize}

  $\CC$ has \myemph{\stable{3}{weakly stable} zero types} if for each $\Gamma$, there exists some type $\zero$ over $\Gamma$ such that for every $\sigma : \Delta \myto \Gamma$, $\zero[\sigma]$ is a zero type over $\Delta$.

   A split comprehension category $\CC$ has \myemph{\stable{0}{strictly stable} zero types} if it is equipped with functions giving for each $\Gamma$ a zero type $\zero_\Gamma$, and for each $C \in \TT(\Gamma.\zero_\Gamma)$ a section, both commuting strictly with reindexing.
\end{definition}

\begin{lemma} \label{lemma:zero-types}
If $\CC$ satisfies condition \LF\ and has \stable{3}{weakly stable} zero types, then $\CC_!$ has \stable{0}{strictly stable} zero types. \qed
\end{lemma}

\paragraph{Unit types}

\begin{definition}
  Given  $\Gamma \in \CC$, a \myemph{unit type over $\Gamma$} consists of:
  \begin{itemize}
  \item a type $\unit \in \TT(\Gamma)$;
  \item a section $\pt : \Gamma \myto \Gamma.\unit$;
  \item for any type $C \in \TT(\Gamma.\unit)$ and section $d$ of $C[\pt]$, a section $\unitelim_{C,d}$ of $C$, such that $\unitelim_{C,d} \comp \pt = d$.
  \end{itemize}

  $\CC$ has \myemph{\stable{3}{weakly stable} unit types} if for each $\Gamma$, there is some $(\unit,\pt)$ over $\Gamma$ such that for every $\sigma : \Delta \myto \Gamma$, elimination sections can be chosen making $(\unit[\sigma],\pt[\sigma])$ a unit type over $\Delta$.

   A split comprehension category $\CC$ has \myemph{\stable{0}{strictly stable} unit types} if it is equipped with functions giving unit types $\unit_\Gamma$, $\pt_\Gamma$, and elimination sections $\unitelim_{C,d}$, all commuting strictly with reindexing.
\end{definition}

\begin{lemma} \label{lemma:unit-types}
If $\CC$ satisfies condition \LF\ and has \stable{3}{weakly stable} unit types, then $\CC_!$ has \stable{0}{strictly stable} unit types. \qed
\end{lemma}

\paragraph{$\W$-types} \label{sec:w-types}

$\W$-types (also known as inductive types, or types of well-founded trees), are the most powerful of the standard type-constructors \cite[p.~79]{martin-lof:bibliopolis}.

Since their rules involve $\Pi$-types, any kind of $\W$-type structure on a comprehension category $\CC$ must be relative to some form of $\Pi$-type structure on $\CC$.
This dependence introduces an extra subtlety into the definition of weakly stable $\W$-types.
We will therefore consider two different weak forms: a simpler form, assuming that the $\Pi$-types of $\CC$ are pseudo-stable (for instance, if they are categorical exponentials); and a more involved but more general form, allowing that the $\Pi$-types themselves may be only weakly stable .

\begin{definition}
Suppose $\CC$ is equipped with some choice of dependent products.
Given $\Gamma \in \CC$, $A \in \TT(\Gamma)$, $B \in \TT(\Gamma.A)$, a \myemph{$\W$-type} $(\W,\fold,\Welim)$ for $\Gamma, A, B$ consists of:
\begin{itemize}
  \item a type $\W \in \TT(\Gamma)$;
  \item a map $\fold : \Gamma.A.\Pi[B,\W[\chi(A)]] \myto \Gamma.\W$, over $\Gamma$;
  \item together with, for any type $C$ over $\Gamma.\W$ and square of the form
  \begin{align*}
    \begin{tikzpicture}[auto]
      \node (UL) at (0,1.75) {$\Gamma.\W.\Pi[B,\W].\Pi[B, C[\app'_{B,\W}]]$};
      \node (BL) at (0,0) {$\Gamma.\W. \Pi[B,\W]$};
      \node (BR) at (4,0) {$\Gamma.\W,$};
      \node (UR) at (4,1.75) {$\Gamma.\W.C$};
      \draw[->] (UL) to node[mylabel] {$d$} (UR);
      \draw[->] (UL) to (BL);
      \draw[->] (BL) to node[mylabel,swap] {$\fold$} (BR);
      \draw[->] (UR) to (BR);
      \node (label) at (2,0.875) {(1)};
    \end{tikzpicture}
  \end{align*}
  a section $\Welim_{C,d} : \Gamma.\W \myto \Gamma.\W.C$, such that the square
  \begin{align*}
    \begin{tikzpicture}[auto]
      \node (UL) at (0,1.75) {$\Gamma.\W.\Pi[B,\W].\Pi[B, C[\app'_{B,\W}]]$};
      \node (BL) at (0,0) {$\Gamma.\W. \Pi[B,\W]$};
      \node (BR) at (4,0) {$\Gamma.\W,$};
      \node (UR) at (4,1.75) {$\Gamma.\W.C$};
      \draw[->] (UL) to node[mylabel] {$d$} (UR);
      \draw[->] (BL) to node[mylabel] {$\lambda(\Welim_{C,d} \comp \app'_{B,\W})$} (UL);
      \draw[->] (BL) to node[mylabel,swap] {$\fold$} (BR);
      \draw[->] (BR) to node[mylabel,swap] {$\Welim_{C,d}$} (UR);
      \node (label) at (2,0.875) {(2)};
    \end{tikzpicture}
  \end{align*}
  commutes.

  (Here $\app'$ is $\app$ with its arguments flipped; and we suppress several weakenings, so e.g.\ $\Pi[B,\W]$ is strictly speaking $\Pi[B,\W[\chi(A)]]$.)
  \end{itemize}
\end{definition}

For the strictly stable case, of course, the $\Pi$-types too must be strictly stable, in order for the stability equations for $\fold$ and $\Welim$ to typecheck.

\begin{definition}
Suppose $\CC$ has \stable{0}{strictly stable} dependent products.
We say $\CC$ has \myemph{\stable{0}{strictly stable} $\W$-types} (over the given dependent products) if it is equipped with operations providing, for all $\Gamma$, $A$, $B$ as above, a $\W$-type $(\W_{A,B}, \fold_{A,B}, \Welim_{A,B})$, such that for $\sigma : \Gamma' \to \Gamma$ and all appropriate $A$, $B$, $C$, $d$,
  \begin{align*}
    \W_{A,B}[\sigma] & = \W_{A [\sigma],B [\sigma]}\\
    \fold_{A,B}[\sigma] & = \fold_{A [\sigma],B [\sigma]} \\
    \Welim_{A,B;C,d}[\sigma] & = \Welim_{A [\sigma],B [\sigma];C [\sigma],D [\sigma]}.
  \end{align*}
\end{definition}

If we assume pseudo-stable dependent products, then we can give a simple definition of weakly stable $\W$-types, along the same lines as the other definitions of weakly stable constructors so far.

\begin{definition}  \label{def:weakly-stable-w-types-1}
Suppose $\CC$ is equipped with \stable{1}{pseudo-stable} dependent products.
A \myemph{weakly stable} $\W$-type for $\Gamma$, $A$, $B$ (over these dependent products) is $(\W,\fold)$ as above, such that for each $\sigma : \Delta \myto \Gamma$, there is some $\Welim$ making $(\W[\sigma],\fold[\sigma],\Welim)$ a $\W$-type for $\Delta$, $A[\sigma]$, $B[\sigma]$.

(Here the reindexing $\fold[\sigma]$ is taken with domain $\Delta.\W[\sigma].\Pi[B[\sigma],\W[\sigma]]$, which by pseudo-stability is a valid reindexing of $\Gamma.\W.\Pi[B,\W]$.)

We say $\CC$ has \myemph{weakly stable} $\W$-types (over the given dependent products) if, for each $\Gamma$, $A$, $B$, there exists some weakly stable $\W$-type.
\end{definition}

\begin{lemma}
If $\CC$ has pseudo-stable dependent products, and weakly stable $\W$-types over these, then $\CC_!$ has strictly stable $\W$-types over the strictly stable dependent products given by Scholium~\ref{sch:good-pi-types} together with Prop.~\ref{prop:stable-class-from-pseudo-stable}. 
\end{lemma}

In maximal generality, one might not want to assume that the $\Pi$-types are any more than weakly stable themselves.
Defining $\W$-types over these, that can lift to $\CC_!$, is a little more involved.

(In fact, dependent products are pseudo-stable in all examples we know of, so for $\W$-types this more general definition is never really needed.
However, it is illustrative of a more general situation that does arise in practice: constructors that depend on others previously defined, where the earlier ones (identity types, perhaps) are only weakly stable.
See Heuristic~\ref{heuristic:further-rules} below for more on this point.)

\begin{definition} \label{def:weakly-stable-w-types-2}
  Suppose $\CC$ is equipped with a stable class $\GG$ of $\Pi$-types.
  $\CC$ has \myemph{\stable{3}{weakly stable}} $\W$-types over $\GG$, if:
  \begin{itemize}
  \item for each $\Gamma$, $A$, $B$ as above, there is some $\W \in \TT(\Gamma)$ such that
  \item for each $\sigma : \Gamma' \myto \Gamma$, any reindexings $A'$, $B'$, $\W'$ of $A$, $B$, $\W$ along $\sigma$, and any good $\Pi$-type $(\Pi[B',\W'],\app_{B',\W'})$, there is some map $\fold : \Gamma'.A'\Pi[B',\W'] \myto \Gamma'.\W'$, over $\Gamma$, such that
  \item for each further $\sigma' : \Gamma'' \myto \Gamma'$, reindexings $A''$, $B''$, $\W''$, type $C \in \TT(\Gamma''.\W'')$, good $\Pi$-type $\acmarxiv{(\Pi[B'',C[\app'_{B',\W'}[\sigma]]],\app_{B'',C[\app'_{B',\W'}[\sigma]]})}{(\Pi[B'',\,C[\app'_{B',\W'}[\sigma]]],\ \app_{B'',\,C[\app'_{B',\W'}[\sigma]]})}$, and \acmarxiv{each }{}map $d : \Gamma''.A''.\Pi[B',\W'][\sigma].\Pi[B'',C[\app'_{B',\W'}[\sigma]]] \to \Gamma''.\W''.C$ over $\fold[\sigma']$ as in the square (1) above, a section $\Welim$ of $C$, such that
  \item for each further reindexing of everything along some map $\sigma'' : \Gamma''' \myto \Gamma''$, and good $\lambda$-abstraction $\lambda(\Welim[\sigma''] \comp \app'_{B',\W'}[\sigma'][\sigma''])$, the square corresponding to (2) above commutes.   
  \end{itemize}
\end{definition}

As one would hope, these definitions of weakly stable $\W$-types agree:

\begin{proposition}
Suppose $\CC$ is equipped with pseudo-stable $\Pi$-types.
Then $\CC$ has weakly stable $\W$-types in the sense of Def.~\ref{def:weakly-stable-w-types-1} over this pseudo-stable structure if and only if it has weakly stable $\W$-types in the sense of Def.~\ref{def:weakly-stable-w-types-2} over the corresponding stable class defined in Prop.~\ref{prop:stable-class-from-pseudo-stable}. \qed
\end{proposition}

We may now lift weakly stable structure to $\CC_!$, using the techniques established above.
Once again, nothing surprising occurs, and no new subtleties arise.

\begin{lemma} \label{lemma:w-types}
Suppose $\CC$ satisfies \LF, and is equipped with a stable class of $\Pi$-types, and \stable{3}{weakly stable} $\W$-types relative to these.
Then $\CC_!$ carries \stable{0}{strictly stable} $\W$-types, over the strictly stable $\Pi$-types from Scholium~\ref{sch:good-pi-types}. \qed
\end{lemma}

\noindent This completes the proof of Theorem~\ref{thm:main-thm}.  \qed

\counterwithin{theorem}{subsection}

\section{Further notes} \label{sec:notes}

\subsection{Generalization to other rules} \label{sec:further-rules}

We have given Theorem~\ref{thm:main-thm} just for (a selection of) the standard constructors and rules of Martin-L\"o{}f type theory.
However, one of the hallmarks of type theory is its extensibility.
One usually wants to consider at least some other rules and constructors beyond these; a technique applying only to this standard core would be highly limited in its utility.
A word is therefore in order on how the present results extend to rules and constructors beyond those considered above.

\begin{heuristic} \label{heuristic:further-rules}
The definition of \myemph{\stable{0}{strictly stable}} structure and the initiality of the syntactic category extend straightforwardly to all reasonable constructors and rules of type theory.

The definition of \myemph{\stable{3}{weakly stable}} structure seems to extend to all reasonable constructors and rules, in general along the (slightly involved) lines of the case of $\W$-types in Section~\ref{sec:w-types} above.

Given these, the lifting of \stable{3}{weakly stable} structure on $\CC$ to \stable{0}{strictly stable} structure on $\CC_!$ extends straightforwardly to all reasonable finitary rules and constructors, \myemph{except for type equality rules}, exactly along the lines of the cases given in Section~\ref{sec:main-proof}.
The finitariness condition may be removed by strengthening \LF\ appropriately. 
\end{heuristic}

For the case of \stable{0}{strictly stable} structure, this is well-understood and generally believed in the community; but, to our knowledge, no precise statement of it has been formulated.
Given such a formulation, one would hope that the present heuristics could be made precise, and the results of this paper given in some more satisfying generality.
However, setting up such a general framework is beyond the scope of this paper; so for now we treat the general case informally. \\

The case of \stable{3}{weak stability} is less clear; extending the definition is in general rather trickier than most of the cases considered in this paper might suggest.

The complication comes from the dependence of later rules on earlier ones.
Since most standard constructors fall into independent groups, we have ``cheated'' slightly within these groups, and given \stable{3}{weak stability} in a slightly simpler form than the general approach would provide.
In the case of $\W$-types, however, we see the complications that arise with successive dependency of rules.

The inputs to each operation (or existence condition) are interpretations of the premises of the corresponding rule, which may involve previously given constructors.
For instance, the inputs to $\W$-elimination involve $\Pi$-types.
If those previous constructors are only \stable{3}{weakly stable}, then the types in the premises will not have canonically chosen interpretations; so one must quantify over all possible choices, or at least over some reasonable class of choices.

So in general, for each operation, one considers $\CC$ as equipped with a \myemph{stable class} of interpretations, called \myemph{good}; and in the inputs to each operation, one quantifies over all good interpretations of the types involved in the premises.
Since these will in general be derived types, not just primitives, this depends on being able to extend the notion of good interpretations to derived types.

This consideration is subtle enough that, without a precise formulation, we cannot confidently claim that this approach extends to \myemph{all} reasonable rules.
However, for all the rules that we have investigated in this connection, it extends without further complications.

(If a constructor does not appear in subsequent rules, then there is no need to distinguish a class of good interpretations; one may without loss of generality consider all weakly stable interpretations as good, and hence to ask that for any given input, some such interpretation exists.) \\

Once \stable{3}{weak stability} is appropriately defined, the lifting from $\CC$ to $\CC_!$ is generally straightforward, modulo two limitations.

Firstly, with \LF\ in its present form, one can lift only finitary rules, since we have just finite limits with which to construct representing objects for premises.
To lift infinitary rules, one needs to strengthen \LF\ by assuming appropriate larger limits in $\CC$.

Secondly, and less negotiably, the lifting works just for rules whose conslusions are type constructors, term constructors, or term equalities.
It does \myemph{not} work for type equality rules.
Even when some equality of derived types holds strictly in $\CC$, their liftings to $\CC_!$ will almost always have different local universes.

The most notable type equality rules considered in practice are the constructor commutation rules for universes à la Tarski \cite[p.~88]{martin-lof:bibliopolis}, and (in the absence of universes) large eliminators for simple inductive types.
In each case, one may replace these rules with forms not involving type equalities, which will then lift.

Most straightforwardly, one can simply replace the equalities by (terms representing) equivalences of types.

Alternatively, in the case of operations on universes, one may directly equip the results of operations with the appropriate structure.
For instance, if $(\univ,\elt[\univ])$ is a universe, and $+^{\univ} : \univ \times \univ \myto \univ$ is the sum-types operation on $\univ$, then the standard commutation rule states that $\elt[\univ](a +^{\univ} b) = \elt[\univ](a) + \elt[\univ](b)$.
Instead, one may posit inclusion maps and an eliminator directly exhibiting $\elt[\univ](a +^{\univ} b)$ as a sum for $\elt[\univ](a)$ and $\elt[\univ](b)$, independently of any globally defined sum types.
This is interestingly analogous to \stable{3}{weak stability}, replacing an explicit commutation condition by the preservation of the universal property.

In either case, it seems that in replacing type equality rules with these weaker forms, one loses only convenience, not logical strength.
Again, however, this is somewhat heuristic and conjectural: generally believed based on practical experience, but not (to our knowledge) known in any precise form.

\subsection{Applications}

Our main motivating examples are homotopy-theoretic in nature, along the lines of \cite{warren-awodey:homotopy-theoretic-models,warren:thesis,garner-van-den-berg:top-and-simp-models,voevodsky:notes-2011-04}, and similar.
The slogan for all such models is \myemph{types as fibrations}.

Specifically, any weak factorization system (or algebraic wfs) on a finitely complete category $\EE$ gives a comprehension category $\TT \myto \EE$ over $\EE$, with $\TT(\Gamma)$ the category of right maps into $\Gamma$.
(We will refer to the right maps of the factorization system as \myemph{fibrations}.)
When it is clear which weak factorization system is under consideration, we refer to the resulting comprehension category again simply as $\EE$.
Similarly, we write $\EE_f$ for the full sub comprehension category on the fibrant objects of $\EE$.

If (the underlying maps of) fibrations are exponentiable, then the ambient hypothesis \LF\ applies; so to model type theory, we need only show that appropriate \stable{3}{weakly stable} structure exists.
As shown in various recent work---\cite{warren-awodey:homotopy-theoretic-models,warren:thesis,arndt-kapulkin:homotopy-theoretic-models}---this structure often follows from well-known homotopy-theoretic facts.
Various combinations of homotopy-theoretic properties turn out to suffice.
One particularly fruitful combination is the following (the terminology is taken from \cite{arndt-kapulkin:homotopy-theoretic-models}; we modify their definition somewhat):

\begin{definition}
A \myemph{logical weak factorization system on $\EE$} is a weak factorization system on $\EE$, such that:
\begin{enumerate}[(a)]
\item \label{log-mod-cat:fibs-exp} fibrations are exponentiable: for any fibration $p : Y \myto X$, the pullback functor $p^* : \EE/X \myto \EE/Y$ has a right adjoint $\Pi_p$; 
\item \label{log-mod-cat:tcofs-pullback} left maps are preserved by pullback along fibrations (equivalently, fibrations are preserved by exponentiation along fibrations); and
\item \label{log-mod-cat:substitutions} any left map $i$ between fibrations over a common base $\Gamma$ 
  \begin{align*}
    \begin{tikzpicture}[auto]
      \node (one) at (0,1.25) {$A$};
      \node (two) at (2,1.25) {$B$};
      \node (base) at (1,0) {$\Gamma$};
      \draw[->] (one) to (base);
      \draw[->] (two) to (base);
      \draw[->] (one) to node[mylabel] {$i$} (two);
    \end{tikzpicture}
  \end{align*}
  is \myemph{substitution-stable in $\Gamma$}, i.e.\ its pullback along any map $f : \Gamma' \myto \Gamma$ is again a left map.
\end{enumerate}

A \myemph{semi-logical weak factorization system} is as above, but with the conditions required only for fibrations over fibrant bases.

A \myemph{(semi-)logical model structure} is a model structure whose (trivial cofibration, fibration) weak factorization system is (semi-)logical.
\end{definition}

\begin{theorem}\label{thm:soundness}
Suppose $\EE$ is a finitely complete category, with stable finite coproducts, equipped with a weak factorization system.

If the weak factorization system is semi-logical, then $(\EE_f)_{!}$ models type theory with $\Pi$-, $\Sigma$-, unit, $\id$-, and finite sum types.
If it is moreover logical, then so does $\EE_!$.
\end{theorem}

\begin{proof}
The assumption of finite completeness, together with (\ref{log-mod-cat:fibs-exp}), ensures that condition \LF\ applies; so by the results of Section \ref{sec:main-thm} it suffices to construct the desired \stable{3}{weakly stable} structure on $\EE$.
Identity types are constructed as in \cite[Thm.~3.1]{warren-awodey:homotopy-theoretic-models} and \cite[Thm.~2.17]{warren:thesis}
The other structure is along standard categorical lines, with condition (\ref{log-mod-cat:substitutions}) providing \stable{3}{weak stability} (compare \cite[Thm.~26]{arndt-kapulkin:homotopy-theoretic-models}).
\end{proof}

For recognizing (semi-)logical weak factorization systems, it is often convenient to replace (\ref{log-mod-cat:substitutions}) with a simpler equivalent criterion:
\begin{lemma} \label{lemma:left-map-stability-criteria}
  In any weak factorization system satisfying (\ref{log-mod-cat:tcofs-pullback}) above, the following are equivalent:
  \begin{enumerate}[(i)]
  \item \label{item:left-maps-1} Any left map between fibrations over a common base $\Gamma$ is substitution-stable in $\Gamma$.
  \item \label{item:left-maps-2} Any map between fibrations over a \acmarxiv{}{common }base $\Gamma$ has an $(\LL,\RR)$ factorization whose $\LL$-factor is substitution-stable in $\Gamma$.
  \item \label{item:left-maps-3} For any fibration $p : X \fibto \Gamma$, the diagonal map $\Delta_p : X \myto X \times_\Gamma X$ has an $(\LL,\RR)$ factorization whose $\LL$-factor is substitution-stable in $\Gamma$.
  \end{enumerate}

  Moreover, they remain equivalent when restricted to fibrant bases.
\end{lemma}

\begin{proof}
  The implications (\ref{item:left-maps-1}) $\Imp$ (\ref{item:left-maps-2}) $\Imp$ (\ref{item:left-maps-3}) are immediate.
  For the converses, (\ref{item:left-maps-2}) $\Imp$ (\ref{item:left-maps-1}) is a standard retract argument, while (\ref{item:left-maps-3}) $\Imp$ (\ref{item:left-maps-2}) is analogous to the proof of \cite[Lemma~4.2.2]{gambino-garner}.
\end{proof}

A stronger condition, but usually easier to verify when it holds, is:
\begin{lemma} \label{lemma:cofibs-stable-implies-logical}
  Let $\EE$ be a model category in which cofibrations are stable under pullback.  Then condition (\ref{log-mod-cat:substitutions}) holds in $\EE$. 
\end{lemma}

\begin{proof}
  By Ken Brown's lemma, weak equivalences between fibrations over a base are always stable under pullback in the base.
  Thus, if cofibrations are also stable under such pullbacks, so are trivial cofibrations.
\end{proof}

\begin{examples}
  The following is a (non-exhaustive) list of examples to which Theorem~\ref{thm:soundness} is easily seen to apply, using the lemmas above.  (Cf.\ \cite[Exs.~2.16]{shulman:univalence-for-inverse-diagrams}.)
  \begin{enumerate}[1.]
  \item The ``canonical'' model structures on $\cat$ and $\gpd$ are logical (since isofibrations are exponentiable in $\cat$) \cite{joyal-tierney:strong-stacks}.
  \item The usual (Kan/Quillen) and quasicategory (Joyal) model structures on $\ssets$ \cite{joyal:theory-of-quasi-cats} are both logical.
  \item Any right proper Cisinski model structure \cite{cisinski:les-prefaisceaux} is logical.
  \item Any locally cartesian closed right proper simplicial model category in which the cofibrations are monomorphisms is logical \cite[Cor.~2.46]{warren:thesis}.
   (Examples include simplicial presheaves and simplicial sheaves.)
 \item For any cofibrantly generated logical model category $\CC$ and small category $\JJ$, the injective and projective model structures on $\CC^{\JJ}$ are logical.
 \item If $\CC$ is any logical model category, and $\JJ$ is an inverse category, then the Reedy (equivalently, injective) model structure on $\CC^\JJ$ is logical.
   \cite[\textsection11]{shulman:univalence-for-inverse-diagrams} treats this example in detail, showing moreover that univalent universes in $\CC$ lift to univalent universes in $\CC^\JJ$, using the results of the present paper to obtain coherence.
 \end{enumerate}
\end{examples}

More examples of (semi-)logical weak factorization sytems are given (with slightly different terminology) in \cite[\textsection 3.3]{awodey:natural-models}.

\opt{acm}{
  \begin{acks}
  \input{temp-local-universes-acks}
  \end{acks}
}


\acmarxiv
  {\bibliographystyle{ACM-Reference-Format-Journals}}
  {\bibliographystyle{amsalphaurlmod}}
\bibliography{local-universes-bib}

\end{document}